\documentclass[a4paper,reqno,11pt,twoside]{amsart}
\usepackage[all]{xy}
\usepackage{km}
\allowdisplaybreaks

\title{Universal functors on symmetric quotient stacks of Abelian varieties}

\author{Andreas Krug}
\address{Mathematisches Institut, Universit\"at Marburg, Deutschland}
\email{andkrug@mathematik.uni-marburg.de}
\author{Ciaran Meachan}
\address{School of Mathematics and Statistics, University of Glasgow, Scotland}
\email{ciaran.meachan@glasgow.ac.uk}

\begin{document}


\begin{abstract}
We consider certain universal functors on symmetric quotient stacks of Abelian varieties. In dimension two, we discover a family of $\bbP$-functors which induce new derived autoequivalences of Hilbert schemes of points on Abelian surfaces; a set of braid relations on a holomorphic symplectic sixfold; and a pair of spherical functors on the Hilbert square of an Abelian surface, whose twists are related to the well-known Horja twist. In dimension one, our universal functors are fully faithful, giving rise to a semiorthogonal decomposition for the symmetric quotient stack of an elliptic curve (which we compare to the one discovered by Polishchuk--Van den Bergh), and they lift to spherical functors on the canonical cover, inducing twists which descend to give new derived autoequivalences here as well.
\end{abstract}

\maketitle

\section*{Introduction}
The derived category of coherent sheaves on a variety is a fundamental geometric invariant with fascinating and intricate connections to birational geometry, mirror symmetry, non-commutative geometry and representation theory, to name but a few. It is fair to say that derived categories are ubiquitous in mathematics. Just as equivalences between derived categories of different varieties can indicate deep and important connections between the respective varieties, equivalences between a derived category and itself can also reveal underlying structures of a variety that would otherwise remain hidden from view. In particular, the autoequivalence group of the derived category naturally acts on the space of stability conditions and the structure of the group of symmetries manifests itself through certain topological properties of the stability manifold, such as simply-connectedness. Moreover, it is known that derived symmetries of smooth complex projective $K$-trivial surfaces $X$ give rise to birational maps between smooth $K$-trivial birational models of certain moduli spaces $M$ on them (see \cite{bayer2014mmp}), which, in turn, can be used to construct derived autoequivalences of the moduli spaces. Classifying such hidden symmetries when $M$ is a compact hyperk\"ahler variety is a long term goal of ours. 

An alternative, more direct, way of constructing derived autoequivalences for hyperk\"ahler varieties is to use $\IP$-objects (see \cite{huybrechts2006pobjects}) or, more generally, $\IP$-functors (see \cite{addington2011new} and \cite{cautis2012flops}); see 
Section \ref{subsec:P} for details on these notions. The most basic example of a $\IP$-object is the structure sheaf $\cO_M$ of a hyperk\"ahler variety $M$. 

One very interesting source of $\IP$-functors are the universal functors associated to hyperk\"ahler moduli spaces. More precisely, if $X$ is a smooth complex projective $K$-trivial surface and $M$ a moduli space of sheaves on $X$ which is hyperk\"ahler, then the Fourier--Mukai transform:
\[
\FM_\cU:=\pi_{M*}(\pi_X^*(\_)\otimes \cU): \cD(X)\to \cD(M),
\]
induced by the universal sheaf $\cU$ on $X\times M$ is conjectured to be a $\IP$-functor; see \cite[\S1]{addington2011new}. This conjecture is proven when $M$ is the Hilbert scheme $X^{[n]}$ of points on a K3 surface $X$ and some instances where $M$ is deformation equivalent to $X^{[n]}$; see \cite{addington2011new}, \cite{addington2016moduli} and \cite{markman2011integral}. Another important case where this conjecture has been successfully verified is when $M$ is the generalised Kummer variety $K_{n-1}\subset A^{[n]}$ associated to an Abelian surface $X=A$; see \cite{meachan2015derived}. In particular, it is shown that the Fourier--Mukai transform: 
\[
\sF_K: \cD(A)\to \cD(K_{n-1}), 
\]
along the universal family on $A\times K_{n-1}$ is a $\IP^{n-2}$-functor for all $n\ge 3$.    

The key to proving that $\sF_K$ is a $\bbP$-functor is the observation that pull-back: 
\[
m^*:\cD(A)\to\cD(A^{[n]}),
\] 
along the Albanese map $m:A^{[n]}\to A$ is a $\bbP$-functor; see \cite{meachan2015derived}. The Albanese map is isotrivial and the fibres are, by definition, the generalised Kummer variety associated to $A$. In particular, we can view the Hilbert scheme $A^{[n]}$ as a family of generalised Kummer varieties $K_{n-1}$ fibred over $A$. Therefore, it makes sense to regard the $\IP$-functor $m^*$ as a family version of the $\IP$-object $\cO_{K_{n-1}}$. 

This raises the question whether the universal $\IP$-functor $\sF_K:\cD(A)\to\cD(K_{n-1})$ is a fibre of some family $\IP$-functor with target $\cD(A^{[n]})$? In the present paper, we construct such a $\IP$-functor $\cD(A\times A)\to \cD(A^{[n]})$ as a suitable combination of the pull-back $m^*$ and the Fourier--Mukai transform along the universal family of $A^{[n]}$ and study some further properties of this and related functors. 

In view of the conjecture concerning $\IP$-functors on moduli spaces discussed above, it is natural to expect the analogous functor for more general fine moduli spaces $M$ of sheaves on an Abelian surface $A$ to be a $\bbP$-functor. Instead of pursuing this further, we translate our functors to the equivariant side of derived McKay correspondence of Bridgeland, King, Reid \cite{bridgeland2001mckay} and Haiman \cite{haiman2001hilbert}:
\[
\Phi:=\trm{BKR}\circ\trm{Haiman}:\cD(A^{[n]})\xra\sim\cD(\Hilb^{\sym_n}(A^n))\xra\sim\cD_{\sym_n}(A^n),
\]
where the symmetric group $\sym_n$ acts on $A^n$ by permuting the factors, and investigate what happens when we vary the dimension $g$ of the Abelian variety $A$. In the case $g=1$, this yields fully faithful functors and, accordingly, a semiorthogonal decomposition of $\cD_{\sym_n}(A^n)$ which we discuss in the second part of the paper.

\subsection*{Summary of main results} 
Let $A^{[n]}$ be the Hilbert scheme of $n$ points on an Abelian surface $A$ and $m:A^{[n]}\to A$ the Albanese map. We can express $A^{[n]}$ as a moduli space of ideal sheaves on $A$ equipped with a universal sheaf $\cU=\cI_\sZ$ on $A\times A^{[n]}$ where $\sZ\subset A\times A^{[n]}$ is the universal family of length $n$ subschemes of $A$. If $\pi_2:A\times A^{[n]}\to A^{[n]}$ denotes the projection then our main result is the following:

\begin{thm*}[\ref{geometricPfnctr}]
The universal functor:
\[
\sF:=\pi_{2*}((\id_A\times m)^*(\_)\otimes\cU):\cD(A\times A)\to\cD(A^{[n]}),
\]
is a $\bbP^{n-2}$-functor for all $n\ge3$ and thus gives rise to an autoequivalence of $\cD(A^{[n]})$.
\end{thm*}

We show that the restriction $\sF_{A\times\{x\}}:\cD(A)\to\cD(K_{n-1})$ of this functor to any fibre over a point in the second factor coincides with the $\bbP^{n-1}$-functor $\sF_K$ considered in \cite[Theorem 4.1]{meachan2015derived}; in particular, $\sF$ is a family version of $\sF_K$.

When $n=3$, our universal functor $\sF:\cD(A\times A)\to\cD(A^{[3]})$ is spherical and we can directly compare it with a similar spherical functor $\sH: \cD(A\times A)\to \cD(A^{[3]})$; constructed as part of a series of $\IP$-functors by the first author in \cite{krug2014nakajima}, whose image is supported on the exceptional divisor.

\begin{thm*}[\ref{braid}]
The autoequivalences 
of $\cD(A^{[3]})$ associated to the two spherical functors: 
\[
\sF,\sH: \cD(A\times A)\to \cD(A^{[3]}),
\] 
satisfy the braid relation: \[T_{\sF} T_\sH T_{\sF} \simeq T_\sH T_{\sF} T_\sH.\]
\end{thm*}

Restricting this result to the Kummer fourfold $K_2$ recovers the braid relation of \cite[Proposition 5.12(ii)]{krug2014nakajima}.

Throughout the article, we study $\sF$ via the triangle of functors $\sF\to \sF'\to \sF''$ induced by the structure sequence $\cI_\sZ\to\cO_{A\times A^{[n]}}\to\cO_\sZ$ associated to $\sZ\subset A\times A^{[n]}$. That is, we have
\[
 \sF'=\pi_{2*}\left((\id_A\times m)^*(\_)\otimes \cO_{A\times A^{[n]}}\right)\qquad\trm{and}\qquad \sF''=\pi_{2*}\left((\id_A\times m)^*(\_)\otimes \cO_{\sZ}\right).
\]

Now, the case $n=2$ is not covered by Theorem \ref{geometricPfnctr}, but it is still interesting to consider. Indeed, we show that our universal functor $\sF$, as well as $\sF''$, is again spherical and has an intimate relationship with Horja's EZ-construction \cite{horja2005derived}: recall that if $q:E=\bbP(\Omega_A)\to A$ is the $\bbP^1$-bundle associated to the exceptional divisor $E$ inside $A^{[2]}$ and $i:E\hra A^{[2]}$ is the inclusion then, for any integer $k$, the functor: 
\[\sH_k:=i_*(q^*(\_)\otimes\cO_q(k)):\cD(A)\to\cD(A^{[2]}),\] 
is spherical with cotwist $[-3]$ and twist $T_{\sH_k}$, which we call the Horja twist.

\begin{thm*}[\ref{spherical-geometric-F-and-F''}]
The universal functors: 
\[
\sF,\sF'':\cD(A\times A)\to \cD(A^{[2]}),
\] 
are both spherical with cotwist $\left(\begin{smallmatrix}-1&1\\0&1\end{smallmatrix}\right)^*[-1]$ and their induced twists satisfy:
\[
T_\sF\simeq T_{m^*}T_{\sF''}T_{m^*}^{-1}\qquad\trm{and}\qquad T_{\sF''}\simeq T_{\sH_{-1}}^{-1}
(\cO(E/2)\otimes (\_))[1].
\] 
\end{thm*}

We observe that the fibres $\sF_K,\sF_K'': \cD(A)\to \cD(K_1A)$ are precisely the functors which were studied by the authors in \cite{krug2015spherical}, where $\sF_K$ and $\sF_K''$ are shown to be spherical functors with cotwist $(-1)^*[-1]$. In particular, Theorem \ref{spherical-geometric-F-and-F''} generalises the results of \cite{krug2015spherical} from the fibre to the whole family.
\bigskip

All the above results are proved by using the derived McKay correspondence $\Psi: \cD_{\sym_n}(A^n)\xrightarrow \sim \cD(A^{[n]})$ to translate the functors $\sF, \sF',\sF'': \cD(A\times A)\to \cD(A^{[n]})$ to equivariant functors $F,F',F'': \cD(A\times A)\to \cD_{\sym_n}(A^n)$ whose compositions with their adjoints are easier to compute; see Section \ref{section-Pfunctors} for details.

We point out that the functors $F, F', F'': \cD(A\times A)\to \cD_{\sym_n}(A^n)$ are interesting in their own right and their definitions make sense for an Abelian variety $A$ of arbitrary dimension, not just an Abelian surface. In Section \ref{sect:curves}, we study the case when $A=E$ is an elliptic curve.

\begin{thm*}[\ref{g1-Sigma-ff}, \ref{g1-ff} \& \ref{sod}]
For $n\ge 3$, we have fully faithful functors: 
\[\Sigma_n^*: \cD(E)\to \cD_{\sym_n}(E^n)\qquad\trm{and}\qquad F: \cD(E\times E)\to \cD_{\sym_n}(E^n),\] 
where $\Sigma_n:E^n\to E$ is the summation morphism, which give rise to a semiorthogonal decomposition: 
\[
\cD_{\sym_n}(E^n)=\langle \cB_n,F(\cD(E\times E)),\Sigma_n^*(\cD(E))\rangle.
\]  
\end{thm*}

At the moment, we are unable to give a geometric description of the category $\cB_n$. However, comparing this semiorthogonal decomposition to the one of Polishchuk--Van den Bergh \cite[Theorem B]{polishchuk2015semiorthogonal} suggests that further investigation will likely yield interesting results. 

\begin{thm*}[\ref{descent}]
If $\varpi:[E^n/\frA_n]\to[E^n/\sym_n]$ is the double cover induced by the alternating subgroup $\frA_n\lhd\sym_n$, then the functors: 
\[
\varpi^*\Sigma_n^*: \cD(E)\to \cD_{\frA_n}(E^n)\qquad\trm{and}\qquad\varpi^*F:\cD(E\times E)\to\cD_{\frA_n}(E^n),
\]
are spherical and the twists descend to give new autoequivalences of $\cD_{\sym_n}(E^n)$.
\end{thm*}

\subsection*{Acknowledgements} We are very grateful to the anonymous referee of \cite{meachan2015derived} for generously suggesting that the 
$\IP$-functor associated to the generalised Kummer could be extended to one on the Hilbert scheme of points and thus inspiring this work. We also thank Gwyn Bellamy and Joe Karmazyn for helpful comments as well as S\"onke Rollenske and Michael Wemyss for their invaluable guidance and support. 


\section{Preliminaries} 
In this paper, $\cD(X)$ will denote the bounded derived category of coherent sheaves on a smooth complex projective variety $X$. For equivariant versions of this category, we refer the reader to \cite{bernstein1994equivariant}, \cite{bridgeland2001mckay}, \cite{ploog2007equivariant} and \cite{elagin2014equivariant}. In an attempt to make this article self contained, and for convenience, we collect the necessary facts below.

\subsection{Equivariant Sheaves}\label{subsect:equi} 
Let $G$ be a finite group acting on a variety $X$. Then $\cD_G(X)$ denotes the bounded derived category of $G$-equivariant coherent sheaves on $X$. Every object $\cE\in\cD_G(X)$ comes with a $G$-linearisation $\lambda
$, which is a collection of isomorphisms $\lambda_g:\cE\xra\sim g^*\cE$ for all $g\in G$ such that $\lambda_1=\id_\cE$ and $\lambda_{gh}=h^*\lambda_g\circ\lambda_h$, but this will often be suppressed in the notation. If $H<G$ is a subgroup then we have a forgetful functor $\Res^G_H:\cD_G(X)\to\cD_H(X)$. The left (and right) adjoint of the \emph{restriction} functor is given by the \emph{induction} functor: 
\[
\Ind_H^G:\cD_H(X)\to\cD_G(X)\;;\;\cE\mapsto \bigoplus_{[g]\in G/H}g^*\cE,
\]
where the sum runs over a complete set of representatives of the left cosets and the linearisation is given by a combination of the $H$-linearisation of $\cE$ and permutation of the direct summands. 

Furthermore, there is a natural morphism of quotient stacks $\varpi:[X/H]\to[X/G]$ which renders commutative diagrams: 
\[\xymatrix{
\cD([X/H]) \ar[rr]^-{\varpi_*} \ar[d]_-\wr && \cD([X/G]) \ar[d]^-\wr & \cD([X/G]) \ar[rr]^-{\varpi^*} \ar[d]_-\wr && \cD([X/H]) \ar[d]^-\wr\\ 
\cD_H(X) \ar[rr]^-{\Ind^G_H} && \cD_G(X) & \cD_G(X) \ar[rr]^-{\Res^G_H} && \cD_H(X).
}\]
In particular, for all $\cE\in\cD([X/G])$ and $\cF\in\cD([X/H])$, the projection formula (which can be found in \cite[Corollary 4.12]{HRstacks}) asserts the existence of a natural isomorphism $\varpi_*(\varpi^*(\cE)\otimes \cF)\simeq \cE\otimes\varpi_*(\cF)$, which is equivalent to 
\begin{equation}\label{stack-proj-formula}
\Ind^G_H(\Res^G_H(\cE)\otimes \cF)\simeq \cE\otimes\Ind^G_H(\cF).
\end{equation}

Every object $\cL\in\cD_G(X)$ gives rise to a natural endofunctor $\MM_\cL:=(\_)\otimes \cL$ which is an equivalence if $\cL$ is a $G$-equivariant line bundle. Similarly, if $\rho$ is a one-dimensional representation of $G$ 
then we set
\[
\MM_\rho:\cD_G(X)\xra\sim\cD_G(X)\;;\;(\cE,\lambda)\mapsto (\cE,\lambda'),
\]
where $\lambda'$ is the linearisation defined by $\lambda'_g:=\lambda_g\circ\rho(g)$.
For example, if $G$ is the symmetric group $\sym_n$ on $n$ elements and $\fra_n$ is the one dimensional alternating representation which acts by multiplication by the sign of a permutation then its induced autoequivalence is denoted by $\MM_{\fra_n}$. Using this notation, equation \eqref{stack-proj-formula} becomes
\begin{equation}\label{M-stack-proj-formula}
\Ind^G_H\circ\MM_\cF\circ\Res^G_H:=\Ind^G_H(\Res^G_H(\_)\otimes \cF)\simeq (\_)\otimes\Ind_H^G(\cF)=:\MM_{\Ind_H^G(\cF)}.
\end{equation}

Let $f:X\to Y$ be a $G$-equivariant map. Then equivariant pushforward $f_*$ and pullback $f^*$ commute with the functors $\Res$, $\Ind$ and $\MM_\rho$ defined above. That is,
\begin{align}\label{rem:equivariant}
\begin{aligned}
f_*\Res\simeq\Res f_*\;;\; f_*\Ind\simeq\Ind f_*\;;\; f_*\MM_\rho\simeq\MM_\rho\;;\\
f^*\Res\simeq\Res f^*\;;\; f^*\Ind\simeq\Ind f^*\;;\; f^*\MM_\rho\simeq\MM_\rho.
\end{aligned}
\end{align}

If $G$ acts trivially on $X$ then we have a functor $\triv_1^G:\cD(X)\to\cD_G(X)$ which equips every object with the trivial $G$-linearisation. The left (and right) adjoint of $\triv_1^G$ is functor of \emph{invariants} $(\_)^G:\cD_G(X)\to\cD(X)$ which sends a sheaf to its fixed part. There are natural isomorphisms of functors:
\begin{equation}\label{eq:GInd}
\Res^G_H\triv_1^G\simeq\triv_1^H\qquad\trm{and}\qquad(\_)^G\Ind_H^G\simeq(\_)^H.
\end{equation}
Moreover, if $Y$ is another variety on which $G$ acts trivially and $f:X\to Y$ is any morphism, then pushforward $f_*$ and pullback $f^*$ commute with $\triv_1^G$ and $(\_)^G$. That is, we have isomorphisms:
\begin{align}\label{rem:invariant}
\begin{aligned}
f_*\triv_1^G\simeq\triv_1^Gf_*\;;\; f_*(\_)^G\simeq(\_)^Gf_*\;;\\
f^*\triv_1^G\simeq\triv_1^Gf^*\;;\; f^*(\_)^G\simeq(\_)^Gf^*.
\end{aligned}
\end{align}

Suppose $\cE=\bigoplus_{i\in I}\cE_i\in\cD_G(X)$ for some finite index set $I$ and that there is a $G$-action on $I$ which is \emph{compatible} with the $G$-linearisation $\lambda$ on $\cE$ in the sense that $\lambda_g(\cE_i)\simeq \cE_{g(i)}$ for all $i\in I$. 
If $\{i_1,\ldots,i_k\}$ is a set of representatives of the $G$-orbits of $I$ and $G_{i_j}:=\stab_G(i_j)$ is the stabiliser subgroup of the element $i_j$ then we have an isomorphism $\cE=\bigoplus_{j=1}^k\Ind_{G_{i_j}}^G\cE_{i_j}$. In particular, if $G$ acts trivially on $X$ then we can compute invariants using the formula: 
\begin{equation}\label{nontransitive}
 \cE^G\simeq\bigoplus_{j=1}^k \cE_{i_j}^{G_{i_j}}.
\end{equation}
Moreover, if the $G$-action on $I$ is transitive then $\cE=\Ind_{G_i}^G$ and equation \eqref{nontransitive} reduces to 
$\cE^G=\cE_i^{G_i}$ for any $i\in I$; 
see \cite[Lemma 2.2]{danila2001sur} and \cite[Remark 2.4.2]{scala2009cohomology}.

\subsection{Spherical and $\bbP$-functors}\label{subsec:P} 
Let $F:\cA\ra\cB$ be an exact functor between triangulated categories with left adjoint $L$ and right adjoint $R$. Then we define\footnote{
Either by working with Fourier--Mukai transforms or dg-enhancements.} the \emph{twist} $T$ and \emph{cotwist} $C$ of $F$ by the following exact triangles:
\[
FR\xra{\eps}\id_\cB\to T\qquad\trm{and}\qquad C\to\id_\cA\xra{\eta} RF,
\] 
where $\eta$ and $\eps$ are the unit and counit of adjunction.

An exact functor $F:\cA\to\cB$ with left adjoint $L$ and right adjoint $R$ is \emph{spherical} if the cotwist $C$ is an autoequivalence of $\cA$ which identifies the adjoints, that is, $R\simeq CL[1]$. We say that a spherical functor is \emph{split} if $RF\simeq\id_\cA\oplus C[1]$.

If $F:A\to B$ is a spherical functor, then the corresponding twist functor $T$ is an equivalence of $\cB$; see \cite{rouquier2006categorification,addington2011new,anno2013spherical,kuznetsov2015fractional,meachan2016note}.

An exact functor $F:\cA\to\cB$ with left adjoint $L$ and right adjoint $R$ is a \emph{$\bbP^n$-functor} if there is an autoequivalence $D$ of $\cA$, called the $\bbP$-cotwist\footnote{The cotwist $C$ and $\bbP$-cotwist $D$ of $F$ are related by $C[1]=D\oplus D^2\oplus\cdots\oplus D^n$.} of $F$, such that 
\begin{equation}\label{Pfnctr-dfn-condition1}
RF\simeq\id_\cA\oplus D\oplus D^2\oplus\cdots\oplus D^n;
\end{equation}
the composition: 
\[
DRF\hra RFRF \xra{R\epsilon F}RF,
\] 
when written in components 
\[
D\oplus D^2\oplus\cdots\oplus D^n\oplus D^{n+1}\to \id_\cA\oplus D\oplus D^2\oplus\cdots\oplus D^n,
\] 
is of the form 
\[\left(\begin{matrix}
\ast & \ast & \cdots & \ast & \ast\\ 
1 & \ast & \cdots & \ast & \ast\\ 
0 & 1 & \cdots & \ast & \ast\\ 
\vdots & \vdots & \ddots & \vdots & \vdots\\ 
0 & 0 & \cdots & 1 & \ast
\end{matrix}\right);\]
and $R\simeq D^nL$. If $\cA$ and $\cB$ have Serre functors then this last condition is equivalent to $S_\cB FD^n\simeq FS_\cA$. 

Addington \cite[Theorem 4.4]{addington2011new} and Cautis \cite[Proposition 6.6]{cautis2012flops} observed that if $F:\cA\to\cB$ is a $\bbP$-functor then the corresponding twist functor: 
\[
P_F:=\cone(\cone(FHR\xra{f}FR)\to\id_\cB),
\]
where $f:=FHR\hra FRFR\xra{\eps FR-FR\eps}FR$, is an autoequivalence of $\cB$. The functor $P_F$ does not depend on the choice of the morphism $\cone(FHR\xra{f}FR)\to\id_\cB$; see \cite{ATunique} for more details on this.

Because all the $\bbP$-functors encountered in this paper will have $\bbP$-cotwist $D=[-2]$ given by the shift functor, we introduce the notation: 
\[
\llbracket c,d\rrbracket:=[c]\oplus [c+2]\oplus \dots\oplus [d-2]\oplus [d]: \cA\to \cA, 
\] 
for integers $c\le d$ such that $d-c$ is even. For example, if $F$ is a $\bbP^n$-functor with $\bbP$-cotwist $D=[-2]$ then we will abbreviate condition \eqref{Pfnctr-dfn-condition1} simply as $RF\simeq\llbracket-2n,0\rrbracket$.

If $F:\cA\to\cB$ is a $\bbP$-functor with $\bbP$-cotwist $D$ and $\Psi:\cA'\to\cA$ is an equivalence then 
\begin{equation}\label{samePtwist}
\trm{$F\Psi$ is a $\bbP$-functor with $\bbP$-cotwist $\Psi^{-1}D\Psi$ and twist $P_{F\Psi}\simeq P_F$}.
\end{equation}
Similarly, if $\Phi:\cB\to\cB'$ is an equivalence then 
\begin{equation}\label{conjugatePtwist}
\trm{$\Phi F$ is a $\bbP$-functor with $\bbP$-cotwist $D$ and twist $P_{\Phi F}\simeq \Phi P_F\Phi^{-1}$};
\end{equation}
see \cite[Lemma 2.4]{krug2015derived} for more details. Analogously, the same formulae hold for spherical functors: $T_{F\Psi}\simeq T_F$ and $T_{\Phi F}\simeq \Phi T_F\Phi^{-1}$; see \cite[Proposition 13]{addington2013categories}. Note that a $\bbP^1$-functor $F$ is a split spherical functor with $\bbP$-cotwist $D=C[1]$ and twist $P_F\simeq T_F^2$; see \cite[Section 4.3]{addington2011new}. 

The following simple lemma will be used later in the text. 

\begin{lem}\label{lem:Pbasechange}
Let $\cX, \cY, \cZ$ be smooth projective stacks, $\cP\in \cD(\cX\times \cY)$ a kernel, and consider the base change \[F_\cZ:=F\boxtimes \id_{\cD(\cZ)}:=\FM_{\cP\boxtimes \reg_{\Delta_\cZ}}\colon \cD(\cX\times \cZ)\to \cD(\cY\times \cZ).\] If $F:=\FM_{\cP}\colon \cD(\cX)\to \cD(\cY)$ is a $\IP^n$-functor with $\IP$-cotwist $D$, then the base change $F_\cZ\colon \cD(\cX\times \cZ)\to \cD(\cY\times \cZ)$ is a $\IP^n$-functor with $\IP$-cotwist $D_{\cZ}$.
\end{lem}
\begin{proof}
The left and right adjoints of $F_\cZ$ are given by the base changes $L_{\cZ}$ and $R_{\cZ}$ respectively. It is easy to check that the base change of functors is compatible with composition. Hence, the properties 
$R_\cZ F_\cZ\simeq \id\oplus D_\cZ\oplus\dots \oplus D_{\cZ}^n$ and $R_\cZ\simeq D_{\cZ}^nL_{\cZ}$ follow from the analogous properties of $F$.

The unit and counit of adjunction $\eta\colon \id \to RF$ and $\eps\colon FR\to \id$ are defined on the level of the Fourier--Mukai kernels; see \cite{ALadjoints} or \cite{CW}. Hence, their base changes can be defined as $\eta_\cZ:=\eta\boxtimes \id$ and $\eps_\cZ:=\eps\boxtimes \id$, respectively. Now, the identities $\eps F \circ F\eta=\id_F$ and $R\eps\circ \eta R=\id_R$ imply the identities $\eps_\cZ F_\cZ\circ F_\cZ\eta_\cZ=\id_{F_\cZ}$ and $R_\cZ\eps_\cZ\circ \eta_\cZ R_\cZ=\id_{R_\cZ}$. This means that $\eta_\cZ$ and $\eps_\cZ$ are the unit and counit of $F_\cZ$. Hence, the desired property of the monad multiplication \[R_\cZ\eps_{\cZ} F_\cZ\colon R_\cZ F_\cZ R_\cZ F_\cZ\to R_\cZ F_\cZ\] follows from the analogous property of $R\eps F\colon RFRF\to RF$.
\end{proof}

\subsection{Relative Fourier--Mukai transforms}\label{subsec:relativeFM}
Let $f: \cX\to S$ and $g: \cY\to S$ be smooth morphisms between smooth projective varieties, which we regard as families over $S$, and consider the cartesian diagram:
\[\xymatrix@!R@!C@C=0.8pc@R=0.8pc{
X\times Y\ar[rr]^-{\pi_Y}\ar[dd]_-{\pi_X}\ar[dr]^-\jmath && Y \ar'[d][dd]\ar[dr]^-{j_Y}\\
& \cX\times_S\cY\ar[dd]_(0.4){\pi_\cX}\ar[rr]^(0.4){\pi_\cY} && \cY\ar[dd]\\
X\ar'[r][rr]\ar[dr]_-{j_X} && \{s\}\ar[dr]\\
& \cX \ar[rr] && S.
}\]
Then, for any object $\cP\in \cX\times_S \cY$, the \textit{relative Fourier--Mukai transform (over $S$)} is given by
\[
 F:=\FM_\cP=\pi_{\cY*}(\pi_{\cX}^*(\_)\otimes \cP): \cD(\cX)\to \cD(\cY).
\]
The projection formula shows that $F$ is isomorphic to the absolute Fourier--Mukai transform along $\imath_*\cP\in \cD(\cX\times \cY)$ where $\imath:\cX\times_S\cY\to\cX\times\cY$ is the closed embedding. If we fix a point $s\in S$ and set $X:=f^{-1}(\{s\})$, $Y:=g^{-1}(\{s\})$, then there is a canonical closed embedding $\jmath: X\times Y\hookrightarrow \cX\times_S \cY$. The \textit{restriction} (or \textit{fibre} $F_s$) of $F$ over $s\in S$ is the Fourier--Mukai transform along $\jmath^*\cP\in \cD(X\times Y)$, that is,
\begin{equation}\label{fibre-relFM}
 F_s:=\FM_{\jmath^*\cP}:=\pi_{Y*}(\pi_X^*(\_)\otimes \jmath^*\cP): \cD(X)\to \cD(Y).
\end{equation}
Moreover, if $j_X: X\hra \cX$ and $j_Y: Y\hra \cY$ denote the natural embeddings of the fibres then flat base change provides us with natural isomorphisms:
\begin{equation}\label{compatible-relFM}
F j_{X*}\simeq j_{Y*} F_s \qquad\trm{and}\qquad j_Y^* F\simeq F_s j_X^*.
\end{equation}

Similarly, if $F: \cD(\cX)\to \cD(\cY)$ and $G:\cD(\cY)\to \cD(\cZ)$ are relative Fourier--Mukai transforms over $S$ then the usual convolution of kernels together with base change shows that we have a natural isomorphism:
\begin{equation}\label{rescomposition}
G_s\circ F_s\simeq (G\circ F)_s. 
\end{equation}
In particular, since Grothendieck duality ensures that the left and right adjoint of a relative Fourier--Mukai functor $F:\cD(\cX)\to\cD(\cY)$ are again relative Fourier--Mukai transforms, equation \eqref{rescomposition} will be helpful to determine the monad structure of $RF$. 


\section{Abelian Surfaces}\label{section-surfaces}
Let $A$ be an Abelian surface, $A^n$ its cartesian product on which the group $\sym_n$ acts by permutation of the factors, and $A^{[n]}$ the Hilbert scheme of $n$ points on $A$.

\subsection{Derived McKay correspondence}
Recall that Haiman \cite{haiman2001hilbert} constructed an isomorphism $A^{[n]}\simeq \Hilb^{\sym_n}(A^n)$, where $\Hilb^{\sym_n}(A^n)$ is the equivariant Hilbert scheme, that is, the fine moduli space of $\sym_n$-invariant zero-dimensional subschemes $Z\subset A^n$ whose global sections $\H^0(\cO_Z)$ are identified with the regular representation (also known as $\sym_n$-clusters). In particular, there is a universal family $\cZ\subset A^{[n]}\times A^n$ whose projections yield a commutative diagram: 
\begin{equation}\label{BKR-diagram}
\xymatrix{& \cZ \ar[dr]^{p} \ar[dl]_-{q} &\\ 
A^{[n]} \ar[dr]_\mu\ar@/_1pc/[ddr]_m && A^n\ar@/^1pc/[ddl]^{\Sigma_n}
\ar[dl]^\pi\\
&A^{(n)}\ar[d]^(0.4){\bar{\Sigma}_n}&\\
&A.&}
\end{equation}
By the derived McKay correspondence of \cite{bridgeland2001mckay} we get an equivalence:
\[
\Phi:=p_*q^*\triv_1^{\sym_n}: \cD(A^{[n]})\xrightarrow\sim \cD_{\sym_n}(A^n). 
\]
One can conlude easily that the functor: 
\[
 \Psi:= (\_)^{\sym_n}q_*p^*: \cD_{\sym_n}(A^n)\xra\sim \cD(A^{[n]}), 
\]
is an equivalence too\footnote{Note, however, that $\Psi$ is not the inverse of $\Phi$ because we are using $p^*$ instead of $p^!$.}; see \cite[Proposition 2.9]{krug2016remarks}.


Now let $\sZ\subset A\times A^{[n]}$ be the universal subscheme and consider the Fourier--Mukai functor $\FM_{\cO_\sZ}:\cD(A)\to\cD(A^{[n]})$. Then Scala's result \cite[Theorem 16]{scala2009some}, states that we have an isomorphism of functors:
\[
\Phi\circ\FM_{\cO_\sZ}\simeq\FM_{\cK^\bullet[1]},
\]
where $\cK^\bullet$ is the $\sym_n$-equivariant complex: 
\[
0\to \bigoplus_{i=1}^n\cO_{D_i}\to\bigoplus_{|I|=2}\cO_{D_I}\otimes\alt_I\to\dots\to\cO_{D_{\{1,\dots,n\}}}\otimes\alt_n\to0
\]
on $A\times A^n$ (concentrated in degrees one to $n$), $\alt_I$ is the alternating representation of the subgroup $\sym_I\subset \sym_n$, and $D_I:=\bigcap_{i\in I}D_i$ are the partial diagonals defined by
\[
A^n\simeq D_i:=\left\{(x,x_1,\dots,x_n)\mid x_i=x  \right\}\subset A\times A^n.
\]
Equivalently, the terms of the complex can be written as 
\[
 \cK^p=\Ind_{\sym_{[p]}\times \sym_{[p+1,n]}}^{\sym_n} \bigl(\reg_{D_{[p]}}\otimes \alt_{[p]}\bigr)
\]
where $[p]=\{1,\dots, p\}$ and $[p+1,n]=\{p+1,\dots, n\}$. 

Remarkably, if one uses the kernel for the BKRH-equivalence for a Fourier-Mukai functor 
\[
\Psi:\cD_{\sym_n}(A^n)\xra\sim\cD(A^{[n]})
\] 
in the opposite direction (which is not an inverse to $\Phi$) then only the first term of the complex $\cK^\bullet$ survives. More precisely, we have an isomorphism of functors:
\[
\FM_{\cO_\sZ}\simeq\Psi\circ\FM_{\cK^1},
\]
where $\cK^1=\bigoplus_{i=1}^n \cO_{D_i}$; see \cite[Theorem 3.6]{krug2016remarks}. Similarly, if we extend $\cK^\bullet$ by $\cK^0:=\cO_{A\times A^n}$ to $0\to\cO_{A\times A^n}\to \bigoplus_{i=1}^n \cO_{D_i}\to\dots\to0$, and denote the extended complex by $\cK^\bullet$ as well, then we have \[\Phi\circ\FM_{\cI_\sZ}\simeq\FM_{\cK^\bullet}.\] 

From now on, we will fix $\cK$ to be the $\sym_n$-equivariant two-term complex
\[
\cK:=\left(0\to\cO_{A\times A^n}\to \bigoplus_{i=1}^n \cO_{D_i}\to0\right)
\]
on $A\times A^n$ (concentrated in degrees zero and one), where the differential is given by restriction of sections; c.f. \cite[Remark 2.2.1]{scala2009cohomology}. Note that we also have \[\FM_{\cI_\sZ}\simeq\Psi\circ\FM_{\cK}.\]

Finally, the summation morphism $\Sigma_n\colon A^n\to A$ is $\sym_n$-equivariant and pullback 
\begin{equation}\label{SigmaPfnctr}
\Sigma_n^*\triv_1^{\sym_n} \colon \cD(A)\to\cD_{\sym_n}(A^n)\trm{ is a $\IP^{n-1}$-functor with $\bbP$-cotwist $[-2]$.}
\end{equation}  
In particular, this means that we have an isomorphism:
\begin{equation}\label{SigmaMonad}
(\_)^{\sym_n}\Sigma_{n*}\Sigma_n^*\triv_1^{\sym_n} \simeq \llbracket -2(n-1),0 \rrbracket.
\end{equation} 
These statements follow from \cite[Theorem 5.2 \& Lemma 6.4]{meachan2015derived}\footnote{Note that the statement of \cite[Lemma 6.4]{meachan2015derived} is not exactly the same as our statement in equation \eqref{m=PsiSigma} since the equivalence $\Psi$ in \cite{meachan2015derived} is the one that we denote by $\Phi$ here. However, the proof is completely analogous.} where it is shown that pullback along the Albanese map $m:A^{[n]}\to A$ is a $\bbP^{n-1}$-functor and 
\begin{equation}\label{m=PsiSigma}
m^*\simeq\Psi\circ \Sigma_n^*\triv_1^{\sym_n}.
\end{equation}

\subsection{$\bbP$-functors on symmetric quotient stacks of Abelian surfaces}\label{section-Pfunctors}
Now, consider the diagram:
\[\xymatrix@R=1pc{
 && \cD_{\sym_n}(A\times A^n) \ar@/^3pc/[rrr]^-{(\_)\otimes\cK} \ar[rrr]^-{(\_)\otimes\cO_{A\times A^n}} \ar@/_3pc/[rrr]_-{(\_)\otimes\left(\bigoplus_i\cO_{D_i}\right)} &&& \cD_{\sym_n}(A\times A^n) \ar[dd]^-{p_{2*}}\\\\
 \cD(A\times A) \ar[rr]^-{\triv_1^{\sym_n}} && \cD_{\sym_n}(A\times A) \ar[uu]^-{(\id_A\times\Sigma_n)^*} &&& \cD_{\sym_n}(A^n),
}\]
and observe that the triangle: 
\begin{equation}\label{FMtri}
\cK\to\cO_{A\times A^n}\to\bigoplus_i\cO_{D_i}
\end{equation} 
of kernels on $A\times A^n$ induces a triangle of Fourier--Mukai functors: 
\[
F\to F'\to F'',
\] 
where $F$, $F'$ and $F''$ are defined as follows:
\begin{align*}
F&:=p_{2*}\circ \MM_\cK\circ(\id_A\times\Sigma_n)^*\circ\triv_1^{\sym_n},\\
F'&:=p_{2*}\circ \MM_{\cO_{A\times A^n}}\circ(\id_A\times\Sigma_n)^*\circ\triv_1^{\sym_n},\\
F''&:=p_{2*}\circ \MM_{\bigoplus_i\cO_{D_i}}\circ(\id_A\times\Sigma_n)^*\circ\triv_1^{\sym_n}.
\end{align*}
\smallskip

Let $R$, $R'$ and $R''$ denote the right adjoints of $F$, $F'$ and $F''$ respectively. The rest of this section will be spent proving that 
\[
RF\simeq \llbracket -2(n-2), 0\rrbracket:=\cO_{\Delta_{A\times A}}\oplus\cO_{\Delta_{A\times A}}[-2]\oplus\cdots\oplus\cO_{\Delta_{A\times A}}[-2(n-2)].
\]
To do this, we will compute $R'F'$, $R'F''$, $R''F'$, $R''F''$ and then take cohomology of a natural diagram of exact functors to obtain $RF$.
\smallskip

First, we simplify the formula for $F'$. Consider the following diagram: 
\begin{equation}\label{diagram1}
\xymatrix@=1.7pc{
A^n\simeq D_i \ar@/^1pc/[ddrrrr]^-{\id_{A^n}} \ar[ddrr]^(0.6){\iota_i} \ar@/_1pc/[ddddrr]_-{(\pr_i, \Sigma_n)} &&&&\\\\
&& A\times A^n \ar[rr]^-{p_2} \ar[dd]^-{\id_A\times\Sigma_n} && A^n \ar[dd]^-{\Sigma_n}\\\\
&& A\times A \ar[rr] \ar[rr]^-{\pi_2} && A,
}\end{equation}
where $\pi_2$, $\pr_i$ and $p_2$ denote the various projections and $\iota_i:D_i\hra A\times A^n$ is the embedding. Note that the triangles are commutative and the square is Cartesian. In particular, base change around the square in diagram \eqref{diagram1} allows us to rewrite the functor $F'$ as
\begin{equation}\label{F'}
 F'\simeq \Sigma_n^*\pi_{2*}\triv_1^{\sym_n}.
\end{equation}

Now, we rephrase the expression for $F''$ in a more tractable form. Recall that for all $\sigma\in\sym_n$ we have an induced automorphism $\sigma^*\in\Aut(A^n)$ defined by \[\sigma^*(x_1,\dots,x_n):=(x_{\sigma^{-1}(1)},\dots,x_{\sigma^{-1}(n)}).\] Let $\sym_n$ act on $A\times A^n$ by permuting the factors of $A^n$ and by the identity on the first factor, and observe that $\sigma^*\cO_{D_i}\simeq\cO_{D_{\sigma(i)}}$ for all $\sigma\in\sym_n$. If we use the notation \[[u,v]:=\{u,u+1,\dots,v\}\subset\bbN\] for positive integers $u\le v$, then $D_i$ is invariant with respect to the action of $\sym_{[1,n]\backslash \{i\}}\simeq\sym_{n-1}$. In particular, $\cO_{D_1}$ carries a natural linearisation by the group $\sym_{n-1}\simeq \sym_{[2,n]}$ and a set of representatives for the left cosets of the subgroup $\sym_{[2,n]}<\sym_n$ is given by the transpositions $\{\tau_{1i}=(1\,\,i)\}_{i=1}^n$, where $\tau_{11}$ denotes the trivial permutation. Therefore, by definition of the induction functor, we have 
\begin{equation}\label{D1}
\Ind_{\sym_{n-1}}^{\sym_n}\cO_{D_1}:=\bigoplus_{i=1}^n \tau_{1i}^*\cO_{D_1}\simeq \bigoplus_{i=1}^n \cO_{D_i}.
\end{equation} 
Plugging this into the equivariant projection formula \eqref{M-stack-proj-formula} gives:
\begin{equation}\label{MD1}
\MM_{\Ind_{\sym_{n-1}}^{\sym_n}\cO_{D_1}}\simeq \Ind_{\sym_{n-1}}^{\sym_n}\MM_{\cO_{D_1}}\Res^{\sym_n}_{\sym_{n-1}}.
\end{equation} 

Putting all this together, we get
\begin{align*}
F''&:=p_{2*}\circ \MM_{\bigoplus_i\cO_{D_i}}\circ(\id_A\times\Sigma_n)^*\circ\triv_1^{\sym_n}\\
&\simeq p_{2*}\MM_{\Ind_{\sym_{n-1}}^{\sym_n}\cO_{D_1}}(\id_A\times \Sigma_n)^*\triv_1^{\sym_{n}} \tag{by \eqref{D1}}\nonumber\\
&\simeq p_{2*}\Ind_{\sym_{n-1}}^{\sym_n}\MM_{\cO_{D_1}}\Res^{\sym_n}_{\sym_{n-1}}(\id_A\times \Sigma_n)^*\triv_1^{\sym_{n}} \tag{by \eqref{MD1}}\nonumber\\
&\simeq p_{2*}\Ind_{\sym_{n-1}}^{\sym_n}\MM_{\cO_{D_1}}(\id_A\times \Sigma_n)^*\Res^{\sym_n}_{\sym_{n-1}}\triv_1^{\sym_{n}} \tag{$(\id_A\times\Sigma_n)$ is $\sym_n$-equivariant}\nonumber\\
&\simeq p_{2*}\Ind_{\sym_{n-1}}^{\sym_n}\MM_{\cO_{D_1}}(\id_A\times \Sigma_n)^*\triv_1^{\sym_{n-1}} \tag{since $\Res_{\sym_{n-1}}^{\sym_n}\triv_1^{\sym_n}\simeq\triv_1^{\sym_{n-1}}$}\nonumber\\
&\simeq \Ind_{\sym_{n-1}}^{\sym_n}p_{2*}\MM_{\cO_{D_1}}(\id_A\times \Sigma_n)^*\triv_1^{\sym_{n-1}} \tag{since $p_2$ is $\sym_n$-equivariant}\\
&\simeq \Ind_{\sym_{n-1}}^{\sym_n}p_{2*}\iota_{1*}\iota_1^*(\id_A\times\Sigma_n)^*\triv_1^{\sym_{n-1}}\tag{by projection formula}\nonumber\\
&\simeq \Ind_{\sym_{n-1}}^{\sym_n}\iota_1^*(\id_A\times\Sigma_n)^*\triv_1^{\sym_{n-1}}\tag{since $p_2\circ\iota_1\simeq\id_{A^n}$}\\
&\simeq \Ind_{\sym_{n-1}}^{\sym_n}(\pr_1,\Sigma_n)^*\triv_1^{\sym_{n-1}}.\tag{since $(\id_A\times\Sigma_n)\circ\iota_1\simeq(\pr_1,\Sigma_n)$}
\end{align*}
In summary, we can rewrite $F''$ as 
\begin{align}\label{F''}
F'' \simeq \Ind_{\sym_{n-1}}^{\sym_n}(\pr_1,\Sigma_n)^*\triv_1^{\sym_{n-1}}.
\end{align}

Therefore, the right adjoints of $F'$ and $F''$ are given by
\begin{equation}\label{R'R''}
R'\simeq (\_)^{\sym_n}\pi_2^!\Sigma_{n*}\qquad\trm{and}\qquad
R''\simeq (\_)^{\sym_{n-1}}(\pr_1,\Sigma_n)_*\Res^{\sym_n}_{\sym_{n-1}}.
\end{equation} 

\begin{lem}\label{Lem-g2}
If $A$ is an Abelian surface then we have the following isomorphisms of endofunctors of $\cD(A\times A)$ for all $n\ge3$:
\begin{enumerate}
\item $R'F'\simeq \pi_2^!\pi_{2*}\llbracket -2(n-1),0\rrbracket$,
\item $R'F''\simeq \pi_2^!\pi_{2*}\llbracket -2(n-2),0\rrbracket$,
\item $R''F'\simeq \pi_2^!\pi_{2*}\llbracket -2(n-1), -2\rrbracket$,
\item $R''F''\simeq \pi_2^!\pi_{2*}\llbracket -2(n-2), -2\rrbracket\oplus \llbracket - 2(n-2),0\rrbracket$.
\end{enumerate}
\end{lem}

\begin{proof}
(i): Direct computation with \eqref{F'} and \eqref{R'R''} yields:
\begin{align*}\label{R'F'}
R'F'&=(\_)^{\sym_n}\pi_2^!\Sigma_{n*}\Sigma_n^*\pi_{2*}\triv_1^{\sym_n}\\
&\simeq \pi_2^!(\_)^{\sym_n}\Sigma_{n*}\Sigma_n^*\triv_1^{\sym_n}\pi_{2*}\tag{since $\pi_2$ is $\sym_n$-equivariant}\\
&\simeq \pi_2^!\pi_{2*}\llbracket -2(n-1),0\rrbracket.\tag{by \eqref{SigmaMonad}}
\end{align*}
\smallskip

(ii): Consider the following commutative diagram:
\begin{equation}\label{commutative-diagram}\xymatrix{
*+[l]{A\times A^{n-1}\simeq A^n}\ar@/^0.5pc/[drr]^-{(\pr_1,\Sigma_n)} \ar[d]_-{\id_A\times\Sigma_{n-1}} && \\ 
A\times A \ar[rr]_-{(\pi_1,\Sigma_2)} && A\times A,
}\end{equation}
and observe that $(\id_A\times \Sigma_{n-1})^*\triv_1^{\sym_{n-1}}$ is a $\IP^{n-2}$-functor with $\bbP$-cotwist $[-2]$ by \eqref{SigmaPfnctr} and Lemma \ref{lem:Pbasechange}. Since $(\pi_1,\Sigma_2)$ is an automorphism, we have an induced autoequivalence $(\pi_1,\Sigma_2)^*$ of $\cD(A\times A)$ which implies the composition: 
\begin{align*}
&(\id_A\times\Sigma_{n-1})^*\triv_1^{\sym_{n-1}}(\pi_1,\Sigma_2)^* \\
&\simeq (\id_A\times\Sigma_{n-1})^*(\pi_1,\Sigma_2)^*\triv_1^{\sym_{n-1}}\tag{since $(\pi_1,\Sigma_2)^*$ is $\sym_{n-1}$-equivariant}\\
&\simeq(\pr_1, \Sigma_n)^*\triv_1^{\sym_{n-1}}\tag{since $(\pr_1, \Sigma_n)\simeq(\pi_1,\Sigma_n)\circ (\id_A\times \Sigma_{n-1})$}
\end{align*} 
is a $\IP^{n-2}$-functor with $\bbP$-cotwist $[-2]$ by \eqref{samePtwist}. In particular, we have
\begin{equation}\label{Pn-2}
(\_)^{\sym_{n-1}}(\pr_1,\Sigma_n)_*(\pr_1,\Sigma_n)^*\triv_1^{\sym_{n-1}}\simeq \llbracket -2(n-2),0\rrbracket. 
\end{equation}
Combining this observation with the commutative diagram in \eqref{diagram1} gives:
\begin{align*}
R'F''&=(\_)^{\sym_n}\pi_2^!\Sigma_{n*}\Ind_{\sym_{n-1}}^{\sym_n}(\pr_1,\Sigma_n)^*\triv_1^{\sym_{n-1}}\\
&\simeq(\_)^{\sym_n}\Ind_{\sym_{n-1}}^{\sym_n}\pi_2^!\Sigma_{n*}(\pr_1,\Sigma_n)^*\triv_1^{\sym_{n-1}}\tag{$\pi_2$ and $\Sigma_n$ are $\sym_n$-equivariant}\\
&\simeq (\_)^{\sym_{n-1}}\pi_2^!\Sigma_{n*} (\pr_1,\Sigma_n)^*\triv_1^{\sym_{n-1}}\tag{$(\_)^{\sym_n}\Ind_{\sym_{n-1}}^{\sym_n}\simeq(\_)^{\sym_{n-1}}$}\\
&\simeq (\_)^{\sym_{n-1}}\pi_2^!\pi_{2*}(\pr_1, \Sigma_n)_*(\pr_1,\Sigma_n)^*\triv_1^{\sym_{n-1}}\tag{$\Sigma_n=\pi_2\circ (\pr_1,\Sigma_n)$}\\
&\simeq \pi_2^!\pi_{2*}(\_)^{\sym_{n-1}}(\pr_1, \Sigma_n)_*(\pr_1,\Sigma_n)^*\triv_1^{\sym_{n-1}}\tag{$\pi_2$ is $\sym_{n-1}$-equivariant}\\
&\simeq \pi_2^!\pi_{2*}\llbracket -2(n-2),0\rrbracket.\tag{by \eqref{Pn-2}}
\end{align*}
\smallskip

(iii): Notice that if $L'$ denotes the left adjoint of $F'$ then $L'F''$ is left adjoint to $R''F'$. Moreover, by \cite[Remark 1.31]{huybrechts2006fourier}, we have
\begin{equation}\label{L'=R'[2n-4]}
L'\simeq S_{\cD(A\times A)}^{-1}R'S_{\cD_{\sym_n}(A^n)}^{}\simeq R'[2n-4],
\end{equation}
and if we continue with our convention of identifying Fourier--Mukai kernels with the functors they represent then \cite[Remark 5.8]{huybrechts2006fourier} says that
\begin{equation}\label{R''F'=dual(L'F'')[2n-4]}
R''F'\simeq S_{\cD(A\times A)}(L' F'')^\vee=(L' F'')^\vee[4].
\end{equation}
Thus, we have
 \begin{align*}
 R''F' &\simeq (L' F'')^\vee [4]\tag{by \eqref{R''F'=dual(L'F'')[2n-4]}}\\
 &\simeq (R'F'')^\vee [8-2n] \tag{by \eqref{L'=R'[2n-4]}} \\
 &\simeq \pi_2^!\pi_{2*}\llbracket-2(n-1),-2\rrbracket.\tag{since $(\pi_2^!\pi_{2*})^\vee\simeq \pi_2^!\pi_{2*}[-6]$}
 \end{align*}
\smallskip

(iv): First note that $R''F''\simeq \cE^{\sym_{n-1}}$ with 
\[\cE=(\pr_1,\Sigma_n)_*\Res^{\sym_n}_{\sym_{n-1}}\Ind_{\sym_{n-1}}^{\sym_n}(\pr_1,\Sigma_n)^*\triv_1^{\sym_{n-1}}.\]
In summary, we will find a decomposition $\cE=\bigoplus_{i\in I}\cE_i$ which is compatible with the $\sym_{n-1}$-linearisation and then use \eqref{nontransitive} to compute invariants. 

Recall from our derivation of \eqref{D1} that $\Ind_{\sym_{n-1}}^{\sym_n}(\_)=\bigoplus_{i=1}^n\tau_{1i}^*(\_)$, where $\tau_{1i}$ is the transposition $(1\,\,i)$ and $\sym_{n-1}$ is identified with $\sym_{[2,n]}:=\sym_{\{2,\dots,n\}}$. Furthermore, we have $(\pr_1,\Sigma_n)\circ \tau_{1i}=(\pr_i,\Sigma_n)$ and hence $\tau_{1i}^*(\pr_1,\Sigma_n)^*=(\pr_i,\Sigma_n)^*$. Combining these two observations, we get 
\begin{align*}
\cE\simeq \bigoplus_{i=1}^n\cE_i \quad\text{with}\quad \cE_i =(\pr_1,\Sigma_n)_*(\pr_i,\Sigma_n)^*\triv_1^{\sym_{n-1}},
\end{align*}
and the linearisation of $\cE$ satisfies $\lambda_g(\cE_i)=\cE_{g(i)}$ for all $g\in \sym_{[2,n]}$. 
In other words, the canonical action of $\sym_{[2,n]}\simeq\sym_{n-1}$ on $\{1,\dots, n\}$ is compatible with the linearisation of $\cE$ in the sense of Section \ref{subsect:equi}.
This action has two orbits given by $\{1\}$ and $\{2,\dots,n\}$. If we choose $\{1,n\}$ as a set of representatives of the $\sym_{n-1}$-orbits of $I$ then we have $\stab_{\sym_{n-1}}(1)\simeq\sym_{n-1}$ and $\stab_{\sym_{n-1}}(n)\simeq\sym_{[2,n-1]}\simeq\sym_{n-2}$. Plugging this information into \eqref{nontransitive}, we get 
\[\cE^{\sym_{n-1}}=\cE_1^{\sym_{n-1}}\oplus \cE_n^{\sym_{n-2}}.\]
In other words, since $R''F''=\cE^{\sym_{n-1}}$, we have
\begin{align*}
R''F''&= (\_)^{\sym_{n-1}}(\pr_1,\Sigma_n)_*(\pr_1,\Sigma_n)^*\triv_1^{\sym_{n-1}}\\
&\qquad\oplus(\_)^{\sym_{n-2}}(\pr_1,\Sigma_n)_*(\pr_n,\Sigma_n)^*\triv_1^{\sym_{n-2}}.
\end{align*}

By \eqref{Pn-2}, the first direct summand is equal to $\llbracket -2(n-2),0\rrbracket$ and so it only remains to show that the second summand is given by $\pi_2^!\pi_{2*}\llbracket -2(n-2),-2\rrbracket$. For this, we consider the commutative diagram:
\begin{equation}\label{diagram2}
\xymatrix@=1.7pc{
A^n \ar@/^1pc/[ddrrrr]^-{(\pr_n, \Sigma_n)} \ar[ddrr]^(0.6){(\pr_1, \Sigma_n,\pr_n)} \ar@/_1pc/[ddddrr]_-{(\pr_1, \Sigma_n)} &&&&\\\\
&& A\times A\times A \ar[rr]^-{\tau\circ\pi_{23}} \ar[dd]^-{\pi_{12}} && A\times A \ar[dd]^-{\pi_2}\\\\
&& A\times A \ar[rr] \ar[rr]^-{\pi_2} && A,
}\end{equation}
where $\tau\colon A\times A\to A\times A$ denotes the permutation of the two factors and the square is Cartesian. Before doing the computation, we observe that we can factorise the morphism $(\pr_1, \Sigma_n, \pr_n)$ using another commutative diagram:
\[\xymatrix{
*+[l]{A\times A^{n-2}\times A\simeq A^n} \ar@/_0.5pc/[drrr]_-{(\pr_1,\Sigma_n,\pr_n)\quad} \ar[rrr]^-{\id_A\times\Sigma_{n-2}\times\id_A} &&& A\times A\times A 
\ar[d]^-{\phi\,=\left(\begin{smallmatrix}
1 & 0 & 0\\
1 & 1 & 1\\
0 & 0 & 1
\end{smallmatrix}\right)}\\ 
&&& A\times A\times A.
}\]
Now, since $\Sigma_{n-2}^*\triv_1^{\sym_{n-2}}$ is a $\IP^{n-3}$-functor with $\IP$-cotwist $[-2]$ by \eqref{SigmaPfnctr} and $\phi$ is an automorphism, we can use \eqref{samePtwist} to deduce that $(\pr_1, \Sigma_n, \pr_n)^*\triv_1^{\sym_{n-2}}$ is also a $\bbP^{n-3}$-functor with cotwist $[-2]$. That is, we have
\begin{equation}\label{Pn-3}
(\_)^{\sym_{n-2}} (\pr_1, \Sigma_n, \pr_n)_*(\pr_1, \Sigma_n, \pr_n)^*\triv_1^{\sym_{n-2}}\simeq \llbracket -2(n-3), 0\rrbracket.
\end{equation}
Now, these results combine to give 
\begin{align*}
&(\_)^{\sym_{n-2}}(\pr_1,\Sigma_n)_*(\pr_n,\Sigma_n)^*\triv_1^{\sym_{n-2}}\\
&\simeq (\_)^{\sym_{n-2}}\pi_{12*}(\pr_1,\Sigma_n, \pr_n)_*(\pr_1,\Sigma_n, \pr_n)^*(\tau\circ\pi_{23})^*\triv_1^{\sym_{n-2}}\tag{by \eqref{diagram2}}\\
&\simeq \pi_{12*}(\_)^{\sym_{n-2}}(\pr_1,\Sigma_n, \pr_n)_*(\pr_1,\Sigma_n, \pr_n)^*\triv_1^{\sym_{n-2}}(\tau\circ\pi_{23})^*\tag{by \eqref{rem:invariant}}\\
&\simeq \pi_{12*}(\tau\circ \pi_{23})^*\llbracket -2(n-3),0\rrbracket\tag{by \eqref{Pn-3}}\\
&\simeq \pi_2^*\pi_{2*}\llbracket -2(n-3), 0\rrbracket\tag{by base change around \eqref{diagram2}}\\
&\simeq \pi_2^!\pi_{2*}\llbracket -2(n-2), -2\rrbracket\tag{since $\pi_2^*\simeq\pi_2^![-2]$},
\end{align*}
which completes the proof.
\end{proof}

\begin{thm}\label{g2-Pfnctr}
If $A$ is an Abelian surface then the universal functor 
\[
F:=p_{2*}(\cK\otimes(\id_A\times\Sigma_n)^*(\triv_1^{\sym_n}(\_))):\cD(A\times A)\to\cD_{\sym_n}(A^n)
\] 
is a $\IP^{n-2}$-functor for all $n\ge3$.
\end{thm}

\begin{proof}
For the computation of $RF$, we can use the following commutative diagram of functors:
\begin{equation}\label{lattice}
\xymatrix{
R''F\ar[r] \ar[d] & R''F' \ar[r]\ar[d] & R''F'' \ar[d]\\
R'F \ar[r] \ar[d] & R'F' \ar[r]\ar[d] & R'F'' \ar[d]\\
RF \ar[r] & RF' \ar[r] & RF'',
} \end{equation}
whose rows and columns are exact triangles. In particular, we can plug in the results from Lemma \ref{Lem-g2} and take cohomology to get:
\[\xymatrix{
{\begin{array}{c}\pi_2^!\pi_{2*}[4][-1]\\ \oplus\\ \llbracket-2(n-2),0\rrbracket[-1]\end{array}}\ar[r] \ar[d] & \pi_2^!\pi_{2*}\llbracket -2(n-1), -2\rrbracket \ar[r]\ar[d] & {\begin{array}{c}\pi_2^!\pi_{2*}\llbracket -2(n-2), -2\rrbracket\\ \oplus\\ \llbracket - 2(n-2),0\rrbracket\end{array}} \ar[d]\\
\pi_2^!\pi_{2*}[4][-1] \ar[r] \ar[d] & \pi_2^!\pi_{2*}\llbracket -2(n-1),0\rrbracket \ar[r]\ar[d] & \pi_2^!\pi_{2*}\llbracket-2(n-2),0\rrbracket \ar[d]\\
\llbracket - 2(n-2),0\rrbracket \ar[r] & \pi_2^!\pi_{2*} \ar[r] & \pi_2^!\pi_{2*}\oplus\llbracket-2(n-2),0\rrbracket[1].}\]

To see that the direct summands of the form $\pi_2^!\pi_{2*}[-2\ell]$ really cancel when taking the cone of the morphisms, as suggested by the diagram of triangles, we need to show that the maps $R''F'\to R''F''$, $R'F'\to R'F''$, $R''F'\to R'F'$, $R''F''\to R'F''$ restrict to isomorphisms on these summands.
First, note that the Fourier--Mukai kernel of $\pi_2^!\pi_{2*}$ is $\reg_{\Delta_{24}}[2]$, where 
\[\Delta_{24}=\{(a,b,c,b)\mid a,b,c\in A\}\subset A^4\,.\] 
In particular, since $\Delta_{24}$ is connected, we have, at the level of Fourier--Mukai kernels, $\Hom(\pi_2^!\pi_{2*}, \pi_2^!\pi_{2*})=\IC$. Hence, it suffices to show that the induced maps are nonzero. Moreover, since the functors $F,F',F''$ are relative Fourier--Mukai functors over $A$, it is enough to show that these maps are nonzero on the restriction to a fibre. Indeed, if we apply \eqref{fibre-relFM} to the cartesian diagram:
\[\xymatrix{
A\times A^n\ar[r]^-{p_2}\ar[d]_-{\id_A\times\Sigma_n}&A^n\ar[d]^-{\Sigma_n}\\
A\times A\ar[r]^-{\pi_2}&A,
}\]
then we see that the fibres $F_0,F'_0,F''_0$ coincide with the Fourier--Mukai functors $F_K,F'_K,F''_K:\cD(A)\to\cD_{\sym_n}(N)$ already considered in \cite[Section 6]{meachan2015derived}, where $N:=\Sigma^{-1}_n(0)$. In particular, the non-vanishing of the components $R_K''F_K'\to R_K''F_K''$, $R_K'F_K'\to R_K'F_K''$, $R_K''F_K'\to R_K'F_K'$, $R_K''F_K''\to R_K'F_K''$ has already been established. Now we can apply \eqref{rescomposition} to conclude that our induced maps $(R''F')_0\to (R''F'')_0$, $(R'F')_0\to (R'F'')_0$, $(R''F')_0\to (R'F')_0$, $(R''F'')_0\to (R'F'')_0$, are nonzero too. Thus, we have
\[
RF\simeq \llbracket -2(n-2), 0\rrbracket.
\]

Next, we need to show that $RF[-2]\hra RFRF\xra{R\epsilon F} RF$ induces an isomorphism on the cohomology sheaves $\cH^i$ for all $2\le i\le 2(n-2)$. Since $\Hom(\cO_\Delta,\cO_\Delta)\simeq\bbC$, it is enough to show that the maps are nonzero but this also follows from the fact that they are nonzero on the fibres over $0\in A$; see \cite[Section 6]{meachan2015derived}.

Finally, $R\simeq H^{n-2}L$ follows from the fact that $L\simeq S_{\cD(A\times A)}^{-1}RS_{\cD_{\sym_n}(A^{n})}$ and the Serre functors are given by $S_{\cD(A\times A)}=[4]$ and $S_{\cD_{\sym_n}(A^{n})}=[2n]$.
\end{proof}

\begin{prop}\label{prop:FMcKay}
We have an isomorphism of functors: \[\Psi F\simeq \sF.\] 
\end{prop}
\begin{proof}
Comparing the equivariant and geometric triangles of functors:
\[F\to F'\to F''\qquad\trm{and}\qquad\sF\to\sF'\to\sF'',\]
it will be enough to show that 
\begin{equation}\label{eq:F'McKay}
\Psi F'\simeq \sF'\quad \text{and}\quad \Psi F''\simeq \sF'',
\end{equation}
since, at the level of Fourier--Mukai kernels, we have $\Hom(\sF', \sF'')=\IC$. For the first isomorphism, we consider the cartesian diagram: 
\begin{equation}\label{cartsquare}
\xymatrix{
A\times A^{[n]}\ar[r]^-{\pr_2} \ar[d]_{(\id_A\times m)} & A^{[n]} \ar[d]^m\\
A\times A\ar[r]^-{\pi_2} &  A.
}\end{equation}
Then, we have
\begin{align*}
 \Psi F' &= \Psi\Sigma_n^*\pi_{2*}\triv_1^{\sym_n}\tag{by \eqref{F'}}\\
 &\simeq \Psi\Sigma_n^*\triv_1^{\sym_n}\pi_{2*}\tag{by \eqref{rem:equivariant}}\\
 &\simeq m^*\pi_{2*}\tag{by \eqref{m=PsiSigma}}\\
 &\simeq \pr_{2*}(\id_A\times m)^*\tag{base change round \eqref{cartsquare}}\\
 &=\sF' \tag{since $\sF'=\FM_{\cO_{A\times A^{[n]}}}$}
\end{align*}


The proof of $\Psi F''\simeq \sF''$ is similar to the argument in \cite[Theorem 3.6]{krug2016remarks}. Indeed, the key observation of \cite[Section 4]{haiman1999macdonald} and \cite[Section 3.1]{krug2016remarks} is the relationship between the universal families: 
\[\cZ\subset A^n\times A^{[n]}\qquad\trm{and}\qquad\sZ\subset A\times A^{[n]},\]
where $\sZ\subset A\times A^{[n]}$ is the universal family of length $n$ subschemes. Namely, the projection $\pr_1\times \id_{A^{[n]}}: A^n\times  A^{[n]}\to A\times A^{[n]}$ restricts to a morphism $f:\cZ\to \sZ$ which is the quotient by the induced action of $\sym_{[2,n]}\simeq \sym_{n-1}$ on $\cZ$. Together with the fact that $\Sigma_n\circ p\simeq m\circ q$ from \eqref{BKR-diagram}, we get a commutative diagram:
\begin{equation}\label{PsiFdiagram}
\xymatrix{
\cZ \ar[r]^-p \ar[d]^-f \ar@/_2pc/[dd]_-q& A^n \ar[drr]^-{(\pr_1,\Sigma_n)} &\\
\sZ \ar[r]^-\iota \ar[d]^-e & A\times A^{[n]}\ar^-{\pi_2}[dl] \ar[rr]_-{\id_A\times m} && A\times A,\\
A^{[n]} && 
}\end{equation}
where $\iota$ is the closed embedding. Also, the equivariant projection formula gives: 
\begin{align}\label{eq:ff}
 (\_)^{\sym_{n-1}}f_*f^*\triv_1^{\sym_{n-1}}\simeq \id_\sZ;
\end{align}
see \cite[Lemma 2.1]{krug2016remarks}. Putting all this together, we get
\begin{align*}
 \Psi F''&\simeq (\_)^{\sym_n}q_*p^*\Ind_{\sym_{n-1}}^{\sym_n}(\pr_1,\Sigma_n)^*\triv_1^{\sym_{n-1}}\tag{by \eqref{F''}}\\
 &\simeq (\_)^{\sym_n}q_*\Ind_{\sym_{n-1}}^{\sym_n}p^*(\pr_1,\Sigma_n)^*\triv_1^{\sym_{n-1}}\tag{by \eqref{rem:invariant}}\\
 &\simeq (\_)^{\sym_n}e_*f_*\Ind_{\sym_{n-1}}^{\sym_n}f^*\iota^*(\id_A\times m)^*\triv_1^{\sym_{n-1}}\tag{commutativity of \eqref{PsiFdiagram}}\\
&\simeq e_*(\_)^{\sym_n}\Ind_{\sym_{n-1}}^{\sym_n}f_*f^*\triv_1^{\sym_{n-1}}\iota^*(\id_A\times m)^*\tag{by \eqref{rem:invariant}}\\
&\simeq e_*(\_)^{\sym_{n-1}}f_*f^*\triv_1^{\sym_{n-1}}\iota^*(\id_A\times m)^*\tag{by \eqref{eq:GInd}}\\
&\simeq \pi_{2*}\iota_*\iota^*(\id_A\times m)^*\tag{by \eqref{eq:ff}}\\
&\simeq \pi_{2*}((\id_A\times m)^*(\_)\otimes \cO_{\sZ})\tag{projection formula}\\
&\simeq \sF''. \qedhere
\end{align*}
\end{proof}

\begin{cor}\label{geometricPfnctr}
If we regard $A^{[n]}$ as a fine moduli space of ideal sheaves on $A$ equipped with a universal sheaf $\cU$ on $A\times A^{[n]}$ then
\[\sF:=\pi_{2*}(\cU\otimes(\id_A\times m)^*(\_)):\cD(A\times A)\to\cD(A^{[n]})\]
is a $\bbP^{n-2}$-functor for all $n\ge3$ and thus gives rise to an autoequivalence of $\cD(A^{[n]})$.
\end{cor}

\begin{proof}
By Proposition \ref{prop:FMcKay} we have an isomorphism $\Psi F\simeq \sF$ and so the statement follows from combining Theorem \ref{g2-Pfnctr} with \eqref{conjugatePtwist}. That is, $\sF$ is a $\bbP^{n-2}$-functor with $\bbP$-cotwist $[-2]$.
\end{proof}

\begin{rmk}
Using similar techniques to those of \cite[p.252]{addington2011new} and \cite[Section 5]{krug2015derived}, one can show that the induced twist $P_\sF\in \Aut(\cD(A^{[n]}))$ is not contained in the subgroup of standard autoequivalences, nor does it coincide with any of the other known autoequivalences such as $P_{m^*}$, a Huybrechts--Thomas twist \cite{huybrechts2006pobjects} or an autoequivalences coming from Ploog's construction \cite[Section 3]{ploog2007equivariant}. 
\end{rmk}


\begin{rmk} 
To see that the restriction of $\sF:\cD(A\times A)\to \cD(A^{[n]})$ really coincides with $\sF_K: \cD(A)\to \cD(K_{n-1})$ from \cite[Section 4]{meachan2015derived}, we can apply \eqref{fibre-relFM} to the cartesian diagram in \eqref{cartsquare}. Indeed, $\sF$ is a relative Fourier--Mukai transform over $A$ and the fibres of $A\times A$ and $A^{[n]}$ over the point $0\in A$ are $A\simeq A\times \{0\}$ and $K_{n-1}$, respectively. The identity $\sF_0=\sF_K$ follows from the fact that $\sZ\cap(A\times K_{n-1})=\sZ_K$.
\end{rmk}

\subsection{Braid relations on holomorphic symplectic sixfolds}
In \cite{krug2014nakajima}, the first author discovered a family of $\bbP^{n-1}$-functors 
\[
H_{\ell,n}:\cD_{\sym_{\ell}}(X\times X^{\ell})\to\cD_{\sym_{n+\ell}}(X^{n+\ell})
\] 
for all $n>\ell$ and $n>1$, where $X$ is any smooth quasi-projective surface. This family of functors can be regarded as a categorical lift of (approximately half of) the Nakajima operators $q_{\ell,n}$; see \cite[Sect.\ 1.3]{krug2014nakajima} for more details. 

In this section we will study the spherical functor $H_{1,2}: \cD(A\times A)\to \cD_{\sym_3}(A^3)$, whose image is supported on the big diagonal of $A^3$, and its relation to the spherical functor $F: \cD(A\times A)\to \cD_{\sym_3}(A^3)$ coming from Theorem \ref{g2-Pfnctr}.

As described in \cite[Section 2.6]{krug2014nakajima}, the functor $H:=H_{1,2}$ sits in a triangle of functors $H\to H'\to H''$ with 
\begin{equation}\label{H'}
H'=\Ind_{\sym_2}^{\sym_3}\delta_{[2,3]*} \MM_{\fra_{[2,3]}} \triv_1^{\sym_2}, 
\end{equation}
where $\sym_2$ is identified with $\sym_{[2,3]}$ and $\delta_{[2,3]}: X\times X\to X^3$, $(a,b)\mapsto (b,a,a)$ is the closed embedding of the partial diagonal $D_{[2,3]}\subset X^3$, and 
\begin{equation}\label{H''}
H''= \delta_{[1,3]*} \MM_{\fra_{[1,3]}} \triv_1^{\sym_3}\iota^*,
\end{equation}
where $\iota:X\to X\times X$ is the diagonal embedding.

We use the
following braiding criterion generalising \cite[Theorem 1.2]{seidel2001braid} from spherical
objects to spherical functors. 
For further generalisations, see also \cite{anno2013spherical} and
\cite[Thm.\ 4.15]{ANSpringer}.

\begin{prop}\label{braid-criterion}
Let $F_1, F_2:\cD(X)\to\cD(Y)$ be two spherical functors between bounded derived categories of smooth projective varieties with $R_1F_2\simeq\Phi$ for some autoequivalence $\Phi$ of $\cD(X)$, where $R_1$ denotes the right adjoint of $F_1$. Then the associated twists satisfy the braid relation: \[T_{F_1}T_{F_2}T_{F_1}\simeq T_{F_2}T_{F_1}T_{F_2}.\]
\end{prop}

\begin{proof}
This is a very slight generalisation of \cite[Proposition 5.11]{krug2014nakajima}. Indeed, one only needs to 
replace $H$ with $\tilde{H}:=F_2\Phi^{-1}$ in the proof of \emph{loc. cit.} and use the fact that $T_{\tilde{H}}:=T_{F_2\Phi^{-1}}\simeq T_{F_2}$ for any spherical functor $F_2$ and autoequivalence $\Phi$. 
\end{proof}

\begin{lem}\label{RH}
Let $F,H: \cD(A\times A)\to\cD_{\sym_3}(A^3)$ be the spherical functors described above and let $R$ be the right adjoint of $F$. Then we have an isomorphism $RH\simeq\phi_*[1]$ where $\phi$ is the automorphism given by 
\[\phi: A\times A\xra{\left(\begin{smallmatrix}0&1\\1&2\end{smallmatrix}\right)}A\times A\quad;\quad(x,y)\mapsto(y,x+2y).\]
\end{lem}

\begin{proof}
By \eqref{R'R''}, we have $R'\simeq (\_)^{\sym_3}\pi_2^!\Sigma_{3*}$ and $R''\simeq (\_)^{\sym_{2}}(\pr_1,\Sigma_3)_*\Res^{\sym_3}_{\sym_{2}}$. This, together with \eqref{H'}, yields
\begin{align*}
R'H'&:=(\_)^{\sym_3}\pi_2^!\Sigma_{3*}\Ind_{\sym_2}^{\sym_3}\delta_{[2,3]*} \MM_{\fra_{[2,3]}} \triv_1^{\sym_2} \\
&\simeq(\_)^{\sym_3}\Ind_{\sym_2}^{\sym_3}\pi_2^!\Sigma_{3*}\delta_{[2,3]*} \MM_{\fra_{[2,3]}} \triv_1^{\sym_2} \tag{equivariance}\\
&\simeq(\_)^{\sym_2}\pi_2^!\Sigma_{3*}\delta_{[2,3]*} \MM_{\fra_{[2,3]}} \triv_1^{\sym_2} \tag{$(\_)^{\sym_3}\Ind_{\sym_2}^{\sym_3}\simeq(\_)^{\sym_2}$}\\
&\simeq\pi_2^!\Sigma_{3*}\delta_{[2,3]*} (\_)^{\sym_2}\MM_{\fra_{[2,3]}} \triv_1^{\sym_2} \tag{by \eqref{rem:invariant}}\\
&\simeq0.\tag{since $(\_)^{\sym_2}\MM_{\fra_{[2,3]}} \triv_1^{\sym_2}=0$}
\end{align*}
Similar arguments, using the fact that the invariants $(\_)^{\sym_k}\MM_{\alt_k}\triv_1^{\sym_k}$ of the sign representation vanish for all $k\ge2$, show that we also have $R'H''=0$ and $R''H''=0$.

For $R''H'$, we use a similar argument to the one used for Lemma \ref{Lem-g2}(iv). Indeed, we have $R''H'\simeq \cE^{\sym_2}$ with
\[
\cE:=(\pr_1,\Sigma_3)_*\Res^{\sym_3}_{\sym_{2}}\Ind_{\sym_2}^{\sym_3}\delta_{[2,3]*} \MM_{\fra_{[2,3]}} \triv_1^{\sym_2}, 
\]
which decomposes as $\cE\simeq \bigoplus_{i=1}^3\cE_i$, where
\[
 \cE_i\simeq (\pr_1,\Sigma_3)_*\tau_{1i}^*\delta_{[2,3]*} \MM_{\fra_{[2,3]}} \triv_1^{\sym_2}\simeq(\pr_1,\Sigma_3)_*\tau_{1i*}\delta_{[2,3]*} \MM_{\fra_{[2,3]}} \triv_1^{\sym_2}.
\]

Next, we notice that the natural $\sym_2$-action on $I=\{1,2,3\}$, under the identification $\sym_2\simeq\sym_{[2,3]}$, has two orbits given by $\{1\}$ and $\{2,3\}$. In particular, if we choose $\{1,3\}$ as a set of representatives of the $\sym_2$-orbits of $I$ then we have $\stab_{\sym_2}(1)\simeq\sym_2$ and $\stab_{\sym_2}(3)\simeq 1$. Thus, we can apply \eqref{nontransitive} to get
\begin{align*}
R''H'\simeq\cE_1^{\sym_2}\oplus\cE_3\simeq (\_)^{\sym_2}(\pr_1,\Sigma_3)_*\delta_{[2,3]*} \MM_{\fra_{[2,3]}} \triv_1^{\sym_2} \oplus (\pr_1,\Sigma_3)_*\tau_{13*}\delta_{[2,3]*} .
\end{align*}
The first direct summand is zero because $(\_)^{\sym_2}\MM_{\alt_2}\triv_1^{\sym_2}=0$ and the second evaluates to $\phi_*$. Plugging this information into the diagram
\[\xymatrix{
R''H\ar[r] \ar[d] & R''H' \ar[r]\ar[d] & R''H'' \ar[d]\\
R'H \ar[r] \ar[d] & R'H' \ar[r]\ar[d] & R'H'' \ar[d]\\
RH \ar[r] & RH' \ar[r] & RH'' 
}\]
shows that $RH\simeq\phi_*[1]$.
\end{proof}

\begin{cor}\label{braid}
Let $F,H:\cD(A\times A)\to\cD_{\sym_3}(A^3)$ be the equivariant spherical functors described above and $\cF,\sH:\cD(A\times A)\to\cD(A^{[3]})$ be their geometric versions, that is, $\sF\simeq\Psi F$ and $\sH:=\Psi H$. Then the spherical twists 
\[
T_F, T_H\in \Aut(\cD_{\sym_3}(A^3))\qquad\trm{and}\qquad T_\sF, T_\sH\in \Aut(\cD(A^{[3]}))
\] 
satisfy the braid relations: \[T_FT_HT_F\simeq T_HT_FT_H\qquad\trm{and}\qquad T_\sF T_\sH T_\sF\simeq T_\sH T_\sF T_\sH.\] 
\end{cor}
\begin{proof}
The statement on the geometric side follows from combining Proposition \ref{braid-criterion} with Lemma \ref{RH}; the equivariant statement can be deduced\footnote{Alternatively, one can use a straight-forward generalisation of Proposition \ref{braid-criterion} to equivariant derived categories to prove the equivariant statement directly.} from this using \eqref{conjugatePtwist}. 
\end{proof}

\begin{rmk}
Restriction to the fibre over zero shows that Corollary \ref{braid} is a family version of the braid relation in \cite[Proposition 5.12(ii)]{krug2014nakajima}.
\end{rmk}

\subsection{Two spherical functors on the Hilbert square of an Abelian surface}\label{subsec:Hilbsquare}
In this section, we analyse our triangle of functors $F\to F'\to F''$ when $n=2$; note that this case is not covered by Theorem \ref{g2-Pfnctr}, where the assumption was $n\ge 3$.

\begin{prop}\label{spherical-F''}
$F'':=\Ind_1^{\sym_2}(\pr_1,\Sigma_2)^*:\cD(A\times A)\to\cD_{\sym_2}(A^2)$ is a spherical functor with cotwist $C_{F''}=\left(\begin{smallmatrix}-1&1\\0&1\end{smallmatrix}\right)^*[-1]$ and twist $T_F''=\MM_\alt[1]$.
\end{prop}

\begin{proof}
By \cite[Lemma 5.2]{krug2014nakajima}, we know that $\Ind:=\Ind_1^{\sym_2}$ is a spherical functor with cotwist $C_{\Ind}=\tau^*[-1]$ and twist $T_{\Ind}=\MM_{\alt}[1]$ where $\tau\in\Aut(A^2)$ is the automorphism induced by the transposition $(1\,\,2)$, which interchanges the two factors, and $\alt:=\alt_2$ is the alternating representation of $\sym_2$. Since $(\pr_1,\Sigma_2)$ is an automorphism, we can use \eqref{samePtwist} to see that $F''=\Ind(\pr_1,\Sigma_2)^*$ is a spherical functor with cotwist $C_{F''}=(\pr_1,\Sigma_2)_*\tau^*(\pr_1,\Sigma_2)^*[-1]=\left(\begin{smallmatrix}-1&1\\0&1\end{smallmatrix}\right)^*[-1]$ and twist $T_{F''}=T_{\Ind}=\MM_\alt[1]$.
\end{proof}

For convenience, we will abbreviate $\Sigma_2^*\triv_1^{\sym_2}$ and $(\_)^{\sym_2}\Sigma_{2*}$ to just $\Sigma_2^*$ and $\Sigma_{2*}$, respectively. We will only expand the notation when it is necessary.

\begin{lem}\label{SigmaSigmaF''=F'}
We have an isomorphism of functors: 
\[
\Sigma_2^*\Sigma_{2*}F''\simeq F'.
\]
\end{lem}

\begin{proof}
Note that, for $n=2$, the group $\sym_{n-1}=1$ is trivial. Hence, we have
\begin{align*}
\Sigma_{2*}F''&\simeq (\_)^{\sym_2}\Sigma_{2*}\Ind_{1}^{\sym_2}(\pr_1,\Sigma_1)^*\tag{by \eqref{F''}}\\
&\simeq(\_)^{\sym_2}\Ind_{1}^{\sym_2}\Sigma_{2*}(\pr_1,\Sigma_2)^*\tag{$\Sigma_2$ is $\sym_2$-equivariant}\\
&\simeq\Sigma_{2*}(\pr_1,\Sigma_2)^*\tag{$(\_)^{\sym_2}\Ind_{1}^{\sym_2}\simeq\id$}\\
&\simeq \pi_{2*}(\pr_1,\Sigma_2)_*(\pr_1,\Sigma_2)^*\tag{$\Sigma_2\simeq\pi_2\circ(\pr_1,\Sigma_2)$}\\
&\simeq \pi_{2*}.\tag{$(\pr_1,\Sigma_2)$ is an automorphism}
\end{align*}
It follows that
\begin{align*}
\Sigma_2^*\Sigma_{2*}F''\simeq \Sigma_2^*\pi_{2*}&:=\Sigma_2^*\triv_1^{\sym_2}\pi_{2*}\tag{expanding notation}\\
&\simeq \Sigma_2^*\pi_{2*}\triv_1^{\sym_2}\tag{$\pi_2$ is $\sym_2$-equivariant}\\
&\simeq F'\tag{by \eqref{F'}},
\end{align*}
which completes the proof.
\end{proof}

Recall from \eqref{SigmaPfnctr} that $\Sigma_2^*\triv_1^{\sym_2}:\cD(A)\to\cD_{\sym_2}(A^2)$ is a spherical functor with cotwist $[-2]$. In particular, $T_{\Sigma_2^*}: \cD_{\sym_2}(A^2)\to \cD_{\sym_2}(A^2)$ is an autoequivalence.

\begin{prop}\label{isom-of-fnctrs}
We have an isomorphism of functors: 
\[
F[1]\simeq T_{\Sigma_2^*}F''.
\]
\end{prop}

\begin{proof}
Recall the triangle $F\to F'\to F''$ of functors where $F'\simeq \Sigma_2^*\pi_{2*}\triv_1^{\sym_2}$ and the triangle $\Sigma_2^*\Sigma_{2*}\to\id_{\cD_{\sym_2}(A^2)}\to T_{\Sigma_2^*}$ defining the twist around $\Sigma_2^*$. Now observe that we have a commutative diagram of triangles: 
\[\xymatrix{
\Sigma_2^*\Sigma_{2*}F'' \ar[rr]^-{\eps F''} \ar[d]^\wr && F'' \ar[rr] \ar@{=}[d] && T_{\Sigma_2^*}F'' \ar[d]\\
F' \ar[rr] && F'' \ar[rr] && F[1].
}\]
Indeed, commutativity of the first square follows from $\Hom(F',F'')\simeq\bbC$ and the map $\eps F''$ necessarily being nonzero; if it were zero then we would contradict the fact that $T_{\Sigma_2^*}F''$ is spherical by \eqref{SigmaPfnctr} and \eqref{conjugatePtwist}. 
In particular, since the morphism $F'\to F''$ is nonzero, these facts imply that the composition $\Sigma_2^*\Sigma_{2*}F''\xra\sim F'\to F''$ must agree (up to scale) with $\eps F''$. Therefore, the cones of these two morphisms are isomorphic.
\end{proof}

\begin{cor}\label{spherical-F}
$F:=p_{2*}(\cK\otimes(\id_A\times\Sigma_2)^*(\triv_1^{\sym_2}(\_))):\cD(A\times A)\to\cD_{\sym_2}(A^2)$ is a spherical functor with cotwist $C_F=C_{F''}=\left(\begin{smallmatrix}-1&1\\0&1\end{smallmatrix}\right)^*[-1]$ and twist 
\[
T_F\simeq T_{\Sigma_2^*}^{{\phantom,}}\MM_{\alt}T_{\Sigma_2^*}^{-1}[1].
\]
\end{cor}

\begin{proof}
Since $T_{\Sigma_2^*}^{\phantom,}$ is an autoequivalence and $F''$ is spherical by Proposition \ref{spherical-F''}, we can use \eqref{conjugatePtwist} to see that $T_{\Sigma^*}F''$ must also be spherical with cotwist $C_{F''}$ and twist $T_{\Sigma^*}^{{\phantom,}}T_{F''}T_{\Sigma^*}^{-1}$. Now the claim follows from Proposition \ref{isom-of-fnctrs}, the fact that $T_{F[1]}\simeq T_F$ by \eqref{samePtwist}, and the description of the twist $T_{F''}$ in Proposition \ref{spherical-F''}.
\end{proof}
\smallskip

If we transport Corollary \ref{spherical-F} to the geometric side of the BKRH-equivalence $\Psi:\cD_{\sym_2}(A^2)\xra\sim\cD(A^{[2]})$ then we can relate our spherical twists to the one discovered by Horja. First let us recall the Horja twists in this specific setup. 

\begin{prop}\label{Horja}
Let $A^{[2]}$ be the Hilbert scheme of two points on an Abelian surface $A$ and consider the following diagram:
\[\xymatrix{
E\;\ar[d]_-q \ar@{^{(}->}[r]^-i & A^{[2]}\\
A, &
}\]
where $q:E=\bbP(\Omega_A)\to A$ is the $\bbP^1$-bundle associated to the exceptional divisor inside $A^{[2]}$ and $i:E\hra A^{[2]}$ is the inclusion. Then, for any integer $k$, the functor \[\sH_k:=i_*(q^*(\_)\otimes\cO_q(k)):\cD(A)\to\cD(A^{[2]})\] is spherical and the induced twists satisfy $T_{\sH_k}T_{\sH_{k+1}}\simeq\MM_{\cO(E)}$. 
\end{prop}

\begin{proof}
By \cite[\S 1.2, Example 5']{addington2011new} we know that $i_*:\cD(E)\to\cD(A^{[2]})$ is spherical with cotwist $C_{i_*}\simeq \MM_{\cO_{E}(E)}[-2]\simeq S_E[-5]$ and twist $T_{i_*}\simeq \MM_{\cO(E)}$. Since $\MM_{\cO_q(k)}q^*$ is fully faithful, $\sH_k:=i_*\MM_{\cO_q(k)}q^*$ is spherical with cotwist $S_A[-5]=[-3]$; see \cite[Proposition 2.1]{addington2011new}.

By \cite[Theorem 2.6]{orlov1993projective}, we have a semi-orthogonal decomposition
\[\cD(E)\simeq\langle q^*\cD(A)\otimes\cO_q(k),q^*\cD(A)\otimes\cO_q(k+1)\rangle.\]
Thus, using Kuznetsov's observation \cite[Theorem 11]{addington2013categories}, which is a special case of \cite[Theorem 4.14]{halpernleistner2013autoequivalences}, we see that $\sH_k$ and $\sH_{k+1}$ are both spherical with cotwist $S_A[-5]=[-3]$, and the twists satisfy $T_{\sH_k}T_{\sH_{k+1}}\simeq T_{i_*}\simeq\MM_{\cO(E)}$. 
\end{proof}

\begin{rmk}
Proposition \ref{Horja} is standard but we have included a proof for completeness; see \cite[Example 8.49(iv)]{huybrechts2006fourier} and \cite[p.231]{addington2011new}. It is a special case of more general construction; see \cite{horja2005derived} and compare with \cite[Theorem 1.3]{addington2016twists}.
\end{rmk}

\begin{rmk}
Since $\cO_E(E)\simeq\cO_q(-2)$, we can use projection formula to see that $\sH_k\simeq \MM_{\cO(-kE/2)}\sH_0$ and hence $T_{\sH_k}\simeq \MM_{\cO(-kE/2)}T_{\sH_0}\MM_{\cO(kE/2)}$ by \eqref{conjugatePtwist}.
\end{rmk}

The fact that $\sym_2$ is a cyclic group means we can apply the results of \cite{krug2017derived} in this situation. More precisely, if we view \eqref{BKR-diagram} as a flop diagram of the corresponding global quotient stacks, then a special case of \cite[Corollary 4.27]{krug2017derived} states that we have the following `flop-flop=twist' result:
\begin{equation}\label{flopflop=Horja}
\Psi\Phi=T_{\sH_{-1}}^{-1}.
\end{equation}

\begin{rmk}
In order to translate \cite[Theorem 4.26 \& Corollary 4.27]{krug2017derived} into expressions like \eqref{flopflop=Horja}, we need to set $n=2$ and then make the following notational substitutions: $\cL=\cO(E/2)$, $\chi=\alt$, $\Theta=\sH_0$ and $\Xi=\delta_*\circ\triv_1^{\sym_2}$, where $\delta:A\to A^{(2)}$ is the diagonal embedding. As stated, their `flop-flop=twist' result reads as $\Psi\Phi\simeq T_{\sH_0}\MM_{\cO(-E)}$ but this can easily be manipulated into our statement in \eqref{flopflop=Horja} by using Proposition \ref{Horja} as follows: \[\Psi\Phi\simeq T_{\sH_{-1}}^{-1}T_{\sH_{-1}}T_{\sH_0}\MM_{\cO(-E)}\simeq T_{\sH_{-1}}^{-1}\MM_{\cO(E)}\MM_{\cO(-E)}\simeq T_{\sH_{-1}}^{-1}.\]
It would be interesting to know how \eqref{flopflop=Horja} generalises to higher dimensions.
\end{rmk}


\begin{rmk}
Equation \eqref{flopflop=Horja} should be compared with similar `flop-flop=twist' results obtained in \cite[Theorem A\&B]{addington2016twists} and \cite[Theorem 1.5]{donovan2016noncommutative}. 
\end{rmk}
\smallskip

Now we can return to look at the twists around our spherical functors on the geometric side and conclude this section.

\begin{cor}\label{spherical-geometric-F-and-F''}
The universal functors: 
\[
\sF,\sF'':\cD(A\times A)\to \cD(A^{[2]}),
\] are both spherical with cotwist $\left(\begin{smallmatrix}-1&1\\0&1\end{smallmatrix}\right)^*[-1]$ and their induced twists satisfy:
\[
T_\sF\simeq T_{m^*}T_{\sF''}T_{m^*}^{-1}\qquad\trm{and}\qquad T_{\sF''}\simeq T_{\sH_{-1}}^{-1}\MM_{\cO(E/2)}[1].
\] 
\end{cor}

\begin{proof}
Recall that $\sF\simeq\Psi F$ and $\sF''\simeq\Psi F''$ by Proposition \ref{prop:FMcKay}. Therefore, the fact that $\sF$ and $\sF''$ are spherical follows immediately from the fact that $F$ and $F''$ are spherical and $\Psi$ is an equivalence; see Corollary \ref{spherical-F}, Proposition \ref{spherical-F''} and \eqref{conjugatePtwist}. Moreover, since $m^*\simeq\Psi\Sigma_n^*$ by \cite[Lemma 6.4]{meachan2015derived} and $m^*:\cD(A)\to\cD(A^{[2]})$ is spherical by \cite[Theorem 5.2]{meachan2015derived}, we have the following chain of isomorphisms:
\begin{align*}
\sF[1]&\simeq\Psi F[1]\tag{by Proposition \ref{prop:FMcKay}}\\
&\simeq \Psi T_{\Sigma_2^*}F''\tag{by Proposition \ref{isom-of-fnctrs}}\\
&\simeq T_{m^*}\Psi F''\tag{by \eqref{conjugatePtwist}}\\
&\simeq T_{m^*}\sF''\tag{by \eqref{eq:F'McKay}}
\end{align*}
In particular, we can use \eqref{samePtwist} and \eqref{conjugatePtwist} to deduce:
\[
T_\sF\simeq T_{\sF[1]}\simeq T_{m^*}T_{\sF''}T_{m^*}^{-1}.
\]

For the description of $T_{\sF''}$, we use \cite[Theorem 4.26(i)]{krug2017derived}, which states: 
\begin{equation}\label{KPS426i}
\MM_\alt\Psi^{-1}\simeq\Phi\MM_{\cO(E/2)}.
\end{equation}
Putting this all together yields:
\begin{align*}
T_{\sF''}&\simeq T_{\Psi F''}\tag{by \eqref{eq:F'McKay}}\\
&\simeq\Psi T_{F''}\Psi^{-1}\tag{by \eqref{conjugatePtwist}}\\
&\simeq\Psi \MM_\alt\Psi^{-1}[1]\tag{by Proposition \ref{spherical-F''}}\\
&\simeq \Psi\Phi\MM_{\cO(E/2)}[1]\tag{by \eqref{KPS426i}}\\
&\simeq T_{\sH_{-1}}^{-1}\MM_{\cO(E/2)}[1]\tag{by \eqref{flopflop=Horja}},
\end{align*}
which completes the proof.
\end{proof}

\begin{rmk}
Restriction of $\sF,\sF''$ to the zero fibre over $A$ recovers the spherical functors $\sF_K,\sF_K'': \cD(A)\to \cD(K_1)$ of \cite{krug2015spherical}. Indeed, \cite[Theorem 2]{krug2015spherical} shows that the twists along these two spherical functors can be factorised as a composition of standard autoequivalences and twists along spherical \textit{objects}, which brings these twists into accordance with Bridgeland's conjecture on the group of autoequivalences of a K3 surface. In particular, we have
\begin{equation}\label{eq:twistrelation}
 T_{\sF''_K}\simeq \prod_{i} T_{\cO_{E_i}(-1)}^{-1}\circ M_{\cO_K(E_K/2)}[1],
\end{equation}
where $E_K\subset K_1$ is the exceptional divisor of the (restricted) Hilbert--Chow morphism $K_1\to A/\{\pm 1\}$ which decomposes into the 16 exceptional curves $E_i\simeq \IP^1$ over the $2$-torsion points of $A$. Since $E_K$ is the restriction of 
the exceptional divior $E\subset A^{[2]}$ of the Hilbert--Chow morphism $\mu: A^{[2]}\to A^{(2)}$, we see that our new relation:
\[
T_{\sF''}\simeq T_{\sH_{-1}}^{-1}\MM_{\cO(E/2)}[1],
\]
restricts to \eqref{eq:twistrelation} on the zero fibre. That is, we have obtained a family version of the results in \cite{krug2015spherical}.
\end{rmk}

\begin{rmk}
Note that Proposition \ref{Horja} and the `flop-flop=twist' result \eqref{flopflop=Horja} hold true if we replace the Abelian surface $A$ with a K3 surface $X$. Moreover, for the Hilbert scheme of two points on a K3 surface, Addington's \cite{addington2011new} twist $T_\sF$, around the functor $\sF:=\FM_{\cI_\sZ}$ where $\sZ\subset X\times X^{[2]}$ is the universal subscheme, and Horja's twist $T_{\sH_{-1}}$ satisfy the \emph{braid relation}. Indeed, if we consider Scala's complex 
\[\cK^\bullet=0\to\cO_{X\times X^2}\to \cO_{D_1}\oplus\cO_{D_2}\to\cO_{D_1\cap D_2}\otimes\alt\to0,\]
then we have identities: $\Phi^{-1}\FM_{\cK^\bullet}\simeq\sF$ and $\Phi^{-1}\FM_{\cK^2[2]}\simeq\sH_{-1}$; see Section \ref{section-surfaces} and \cite[Proposition 4.2]{krug2015derived}, respectively. Furthermore, if we consider the triangle:
\[\cK^{\ge2}\to\cK^\bullet\to\cK^{\le1}\qquad\xra{\Phi^{-1}\FM_{(\_)}}\qquad \sH_{-1}[-2]\to \sF\to \Phi^{-1}\FM_{\cK^{\le1}},\]
then \cite[Sections 5.5 \& 5.6]{krug2014nakajima} shows that $\sG:=\Phi^{-1}\FM_{\cK^\le1}\simeq T_{\sH_{-1}}\sF$ is also a spherical functor and any two of $T_\sF,T_\sG,T_{\cH_{-1}}$ satisfy the braid relation and generate the group $\langle T_\sF,T_\sG,T_{\cH_{-1}}\rangle$. Note that the twist around $\sG$ agrees, up to conjugation by Horja, with Addington's twist; indeed, it follows from \eqref{conjugatePtwist} that $T_\sG\simeq T_{\sH_{-1}}T_\sF T_{\sH_{-1}}^{-1}$. Alternatively, we can combine the identity $\Psi\FM_{\cK^{\le1}}\simeq\sF$ with the formula in \eqref{flopflop=Horja} to see that 
\[\sF\simeq\Psi\Phi\sG\simeq T_{\sH_{-1}}^{-1}\sG,\]
and observe that in order to generate all of the hidden symmetries $T_\sF,T_\sG,T_{\cH_{-1}}$, one only needs to take Addington's functor $\sF=\FM_{\cI_\sZ}:\cD(X)\to\cD(X^{[2]})$ and the BKRH-equivalence $\Phi:\cD(X^{[2]})\xra\sim\cD_{\sym_2}(X^2)$. Also notice that because of the identity $T_{\sH_k}=\MM_{\cO(-kE/2)}T_{\sH_0}\MM_{\cO(kE/2)}$, the single Horja twist $T_{\sH_{-1}}$, together with the standard autoequivalences on $\cD(X^{[2]})$, will generate all of the Horja twists $T_{\cH_k}$. 
\end{rmk}


\section{Elliptic Curves}\label{sect:curves}
In this section, we turn our attention to the genus one case. To emphasise this, we change our notation from $A$ to $E$. That is, we focus on the derived category of the symmetric quotient stack $\cD([E^n/\sym_n])\simeq\cD_{\sym_n}(E^n)$ where $E$ is an elliptic curve.

Let $\Sigma_n: E^n\to E$ be the summation map and define $N:=\Sigma^{-1}(0)\subset E^n$ as the fibre over zero. That is,
\begin{align}\label{eq:N}
 N:=\Sigma^{-1}(0)=\{(a_1,\dots,a_n)\mid a_1+\dots+a_n=0\}\subset E^n.
\end{align}
Observe that $N\simeq E^{n-1}$. Moreover, the subvariety $N$ is invariant under the natural action of $\sym_n$ on $E^n$ and the associated quotient stack $[N/\sym_n]$ is usually called the \textit{generalised Kummer stack} associated to $E$ and $n$.

\subsection{Fully faithful functors for symmetric quotient stacks of elliptic curves}\label{section-ff}
\begin{prop}\label{prop:Oexc}
The structure sheaf $\cO_{[N/\sym_n]}$ of the generalised Kummer stack is an exceptional object in $\cD([N/\sym_n])\simeq \cD_{\sym_n}(N)$. This means that 
\[
\Hom^*_{\cD_{\sym_n}(N)}(\cO_N, \cO_N)\simeq\IC.
\]
\end{prop}

\begin{proof}
We have $\Hom^*_{\cD_{\sym_n}(N)}(\cO_N, \cO_N)\simeq \H^*(\cO_N)^{\sym_n}$. Since $N$ is an Abelian variety, we have isomorphisms:
 \[
\H^*(\cO_N)\simeq \H^0(\wedge^*\Omega_N)\simeq \H^0(\wedge^* \Omega_N|_0\otimes_\IC \cO_N)\simeq \wedge^* \Omega_N|_0, 
 \]
where $\Omega_N|_0$ is the Zariski cotangent space of $N$ at $0$. The permutation action of $\sym_n$ on $E^n$ induces the permutation action on $\Omega_{E^n}|_0\simeq\IC^n$. It follows from \eqref{eq:N} that the induced action on $\Omega_N|_0\simeq \IC^{n-1}$ is given by the standard representation $\rho_n$. Now, recall that $\wedge^k\rho_n$ is a non-trivial irreducible representation for all $1\le k\le n-1$; see \cite[Proposition 3.12]{fulton1991representation}. In particular, its invariants must vanish and we get
\[
 \H^*(\cO_N)^{\sym_n}\simeq \bigl(\wedge^*\Omega_N|_0\bigr)^{\sym_n}\simeq \bigl(\wedge^*\rho_n\bigr)^{\sym_n}=\IC.\qedhere
\]
\end{proof}

\begin{lem}\label{g1-Sigma-pushforward}
If $\Sigma_n: E^n\to E$ is the summation map then we have 
\[
\bigl(\Sigma_{n*}\cO_{E^n}\bigr)^{\sym_n}\simeq \cO_E.
\]
\end{lem}

\begin{proof}
Since the fibres of $\Sigma_n$ are connected, it is sufficient to show the vanishing of the invarants $\bigl(R^i\Sigma_{n*}\cO_{E^n}\bigr)^{\sym_n}$ of the higher push-forwards for $i>0$. Let $x\in E$ be a point and $\iota_x: \{x\}\hookrightarrow E$ its inclusion. Choosing an $y\in E$ with $ny=x$, we get an $\sym_n$-equivariant isomorphism: 
\[
 N\xrightarrow \sim \Sigma_n^{-1}(x)\;;\; (a_1,\dots, a_n)\mapsto (a_1+y,\dots, a_n+y).
\]
Thus, by Proposition \ref{prop:Oexc}, we have $\H^i( \Sigma_n^{-1}(x), \cO_{\Sigma_n^{-1}(x)})^{\sym_n}=0$ for $i>0$. Now, by flat base change we get
\[
 \iota_x^*\bigl(R^i\Sigma_{n*}\cO_{E^n}\bigr)^{\sym_n}\simeq \H^i( \Sigma_n^{-1}(x), \cO_{\Sigma_n^{-1}(x)})^{\sym_n}=0,
\]
for every $x\in E$, which implies the assertion.
\end{proof}

\begin{thm}\label{g1-Sigma-ff}
The functor
$
\Sigma_n^*\triv_1^{\sym_n}:\cD(E)\to\cD_{\sym_n}(E^n)
$ 
is fully faithful.
\end{thm}

\begin{proof}
The right adjoint is given by $(\_)^{\sym_n}\Sigma_{n*}:\cD_{\sym_n}(E^n)\to\cD(E)$. Thus, by projection formula and Lemma \ref{g1-Sigma-pushforward} we get 
\begin{equation}\label{EllipticSigmaMonad}
(\_)^{\sym_n}\Sigma_{n*}\Sigma_n^*\triv_1^{\sym_n}\simeq\id_{\cD(E)},
\end{equation} 
as required.
\end{proof}

\begin{rmk}
Notice that Theorem \ref{g1-Sigma-ff} can only work in the equivariant setting. Indeed, pullback along the Albanese map $m:E^n\to E$ can never be fully faithful because $E^n$ is Calabi-Yau and hence its derived category cannot admit a non-trivial semiorthogonal decomposition. However, the canonical bundle of the quotient stack $[E^n/\sym_n]$ is given by $\cO_{E^n}\otimes \fra_n$; see \cite[Lemma 5.10]{krug2015equivariant}, and so this means that it is possible for $\cD_{\sym_n}(E^n)$ to admit interesting semiorthogonal decompositions. We will see an example of one such decomposition in the next section.
\end{rmk}


Recall from Section \ref{section-Pfunctors} that we have a triangle of functors $F\to F'\to F''$ induced by the $\sym_n$-equivariant triangle $\cK\to \cO_{E\times E^n}\to\bigoplus_{i=1}^n\cO_{D_i}$. In particular, we have
\[
F'\simeq \Sigma_n^*\pi_{2*}\triv_1^{\sym_n}\qquad\trm{and}\qquad F''\simeq \Ind_{\sym_{n-1}}^{\sym_n}(\pr_1,\Sigma_n)^*\triv_1^{\sym_{n-1}}.
\]

\begin{lem}\label{Lem-g1}
If $E$ is an elliptic curve then we have the following isomorphisms of endofunctors of $\cD(E\times E)$ for all $n\ge3$:
\begin{enumerate}
\item $R'F'\simeq \pi_2^!\pi_{2*}$,
\item $R'F''\simeq \pi_2^!\pi_{2*}$,
\item $R''F'\simeq \pi_2^!\pi_{2*}[-2]$,
\item $R''F''\simeq \pi_2^!\pi_{2*}[-2]\oplus \id_{\cD(E\times E)}$.
\end{enumerate}
\end{lem}

\begin{proof}
The proof is analogous to that of Lemma \ref{Lem-g2} using Theorem \ref{g1-Sigma-ff} instead of \eqref{SigmaMonad}. Indeed, if one replaces the Abelian surface $A$ with an elliptic curve $E$ in the proof of Lemma \ref{Lem-g2}, and the formula in \eqref{SigmaMonad} with the one in \eqref{EllipticSigmaMonad}, then the arguments go through verbatim. In particular, we see that $(\pr_1,\Sigma_n)^*\triv_1^{\sym_{n-1}}$ is now a fully faithful functor rather than a $\bbP^{n-2}$-functor. That is, equation \eqref{Pn-2} becomes 
\begin{equation}\label{pr1sigma-ff}
(\_)^{\sym_{n-1}}(\pr_1,\Sigma_n)_*(\pr_1,\Sigma_n)^*\triv_1^{\sym_{n-1}}\simeq \id_{\cD(E\times E)}
\end{equation}
in the case of an elliptic curve.
\end{proof}

\begin{thm}\label{g1-ff}
If $E$ is an elliptic curve then the universal functor: 
\[F:=p_{2*}(\cK\otimes(\id_A\times\Sigma_n)^*(\triv_1^{\sym_n}(\_))):\cD(E\times E)\to\cD_{\sym_n}(E^n),\] 
is fully faithful for all $n\ge3$.
\end{thm}

\begin{proof}
Take the information from Lemma \ref{Lem-g1}, feed it into the diagram \eqref{lattice} and take cohomology to get $RF\simeq\id_{\cD(E\times E)}$.
\end{proof}

\begin{rmk}
Notice that the fully faithful functor $F: \cD(E\times E)\hookrightarrow \cD_{\sym_n}(E^n)$ restricts to a fully faithful functor $\cD(E)\hookrightarrow\cD_{\sym_n}(N)$; see Section \ref{subsec:relativeFM}.  
\end{rmk}

\subsection{The induced semiorthogonal decomposition}\label{subsec:sod}
Recall that a semiorthogonal decomposition of a triangulated category $\cA$ is a sequence $\cA_1,\dots,\cA_n\subset\cA$ of full admissible subcategories such that $\Hom(\cA_j,\cA_i)=0$ for all $i<j$ and the smallest triangulated category containing all the $\cA_i$ is $\cA$ itself; we say that $\cA$ is generated by the $\cA_i$ and denote a semiorthogonal decomposition of $\cA$ as $\cA=\langle\cA_1,\dots,\cA_n\rangle$.

As before, we abbreviate $\Sigma_n^*\triv_1^{\sym_n}$ and $(\_)^{\sym_n}\Sigma_{n*}$ to just $\Sigma_n^*$ and $\Sigma_{n*}$, respectively; expanding the notation when necessary.  

\begin{lem}\label{SigmaF=0}
Let $F,\Sigma_n^*:\cD(E\times E)\to\cD_{\sym_n}(E^n)$ be the fully faithful functors from Theorem \ref{g1-ff} and Theorem \ref{g1-Sigma-ff}, respectively. Then we have 
\[
\Sigma_{n*}F\simeq 0. 
\]
\end{lem}

\begin{proof}
We will apply $\Sigma_{n*}$ to the triangle $F\to F'\to F''$ and observe that 
\[\Sigma_{n*}F'\simeq\pi_{2*}\simeq\Sigma_{n*}F''\implies \Sigma_{n*}F\simeq0.\]
Indeed, we have
\begin{align*}
\Sigma_{n*}F'&:= (\_)^{\sym_n}\Sigma_{n*}\Sigma_n^*\pi_{2*}\triv_1^{\sym_n}\tag{expanding notation and \eqref{F'}}\\
&\;\simeq (\_)^{\sym_n}\Sigma_{n*}\Sigma_n^*\triv_1^{\sym_n}\pi_{2*}\tag{since $\pi_2$ is $\sym_n$-equivariant}\\
&\;\simeq\pi_{2*}\tag{by \eqref{EllipticSigmaMonad}},
\end{align*}
and 
\begin{align*}
\Sigma_{n*}F''&:=(\_)^{\sym_n}\Sigma_{n*}\Ind_{\sym_{n-1}}^{\sym_n}(\pr_1,\Sigma_n)^*\triv_1^{\sym_{n-1}}\tag{expanding notation and \eqref{F''}}\\
&\simeq\pi_{2*}(\_)^{\sym_{n-1}}(\pr_1,\Sigma_n)_*(\pr_1,\Sigma_n)^*\triv_1^{\sym_{n-1}}\tag{by proof of Lemma \ref{SigmaSigmaF''=F'}}\\ 
&\simeq\pi_{2*}\tag{by \eqref{pr1sigma-ff}}.
\end{align*}

In order to conclude using the triangle $\Sigma_{n*}F\to\Sigma_{n*} F'\to \Sigma_{n*}F''$, it is only left to show that the map $\Sigma_{n*} F'\to \Sigma_{n*}F''$ induces an automorphism of $\pi_*$. Since the Fourier--Mukai kernel $\cO_{\Gamma_\pi}$ of $\pi_*$ is simple in the sense that $\End(\cO_{\Gamma_\pi})=\IC$, it is sufficient to show that $\Sigma_{n*} F'\to \Sigma_{n*}F''$ is non-zero. To see this we plug the object $\cO_{E\times E}$ into both functors to get
\[
F'(\cO_{E\times E})\simeq \H^*(\cO_E)\otimes \cO_{E^n}\qquad\trm{and}\qquad F''(\cO_{E\times E})\simeq \oplus_{i=1}^n \cO_{E^n}.  
\]
The degree zero part of the morphism $F'(\cO_{E\times E})\to F''(\cO_{E\times E})$, induced by the restrcition map $\cO_{E^n\times E}\to \oplus_i\cO_{D_i}$ (see \eqref{FMtri}), is given by $\prod_{i=1}^n \id: \cO_{E^n}\to \oplus_{i=1}^n \cO_{E^n}$. Applying $(\_)^{\sym_n}\Sigma_{n*}$ induces the identity map on $\cO_E=\mathcal H^0(\pi_{2*}(\cO_{E\times E}))$. 
\end{proof}

\begin{cor}\label{sod}
There is a semiorthogonal decomposition
\[
\cD_{\sym_n}(E^n)=\langle \cB_n,F(\cD(E\times E)),\Sigma_n^*(\cD(E))\rangle.
\]
\end{cor}

\begin{proof}
If $\cA_1:=F(\cD(E\times E))$ and $\cA_2:=\Sigma_n^*(\cD(E))$ then Lemma \ref{SigmaF=0} shows that $\cA_1\subset\cA_2^\perp$, that is, $\Hom(\cA_2,\cA_1)=0$. In other words, we have a semiorthogonal decomposition $\cD_{\sym_n}(E^n)=\langle \cB_n,\cA_1,\cA_2)\rangle$, where $\cB_n:=\langle \cA_1,\cA_2\rangle^\perp$.
\end{proof}

\begin{rmk}
Despite not having a precise description for the component $\cB_n$, we do have other semiorthogonal decompositions of $\cD_{\sym_n}(E^n)$, due to \cite{krug2014nakajima} and \cite{polishchuk2015semiorthogonal}, which we can compare ours to; the semiorthogonal decompositions of \emph{loc. cit.} work for an arbitrary smooth projective curve, whereas the decomposition of Corollary \ref{sod} is specific to the case of elliptic curves. More precisely, the components of the semiorthogonal decomposition appearing in \cite[Theorem B]{polishchuk2015semiorthogonal} are given by \[\cD(E^{(\nu_1)}\times E^{(\nu_2)}\times \dots\times E^{(\nu_n)});\] one such piece for every partition $1^{\nu_1}2^{\nu_2}\cdots n^{\nu_n}$ of $n$, where  $1^{\nu_1}2^{\nu_2}\cdots$ stands for the partition $(1,\dots, 1,2,\dots, 2,\dots)$ with $1$ occurring $\nu_1$ times, $2$ occuring $\nu_2$ times, and so on. In particular, we have $1\cdot \nu_1+2\cdot \nu_2+\dots +n\cdot \nu_n=n$. Furthermore, this decomposition contains one component equivalent to $\cD(E\times E)$ (corresponding to $1^1(n-1)^1$) and one equivalent to $\cD(E)$ (corresponding to $n^1$). However, the embeddings of these components are fundamentally different from our embeddings $F$ and $\Sigma_n$, respectively. Indeed, the objects in images of the embeddings in \cite{polishchuk2015semiorthogonal} are all supported on partial diagonal whereas in 
$F(\cD(E\times E))$ and $\Sigma_n^*(\cD(E))$ there are objects supported on the whole $E^n$. 

In view of the above, it seems natural to expect that the component $\cB_n$ of our semiorthogonal decomposition in Corollary \ref{sod} can be refined to a semiorthogonal decomposition consisting of one piece equivalent to $\cD(E^{(\nu_1)}\times E^{(\nu_2)}\times \dots\times E^{(\nu_n)})$ for every partition of $n$ besides $1^1(n-1)^1$ and $n^1$. The corresponding fully faithful embeddings would then also be promising candidates for further $\IP$-functors if we go back from the elliptic curve $E$ to an Abelian surface $A$; we plan to return to this in future work.
\end{rmk}

\subsection{Alternating quotient stacks of elliptic curves and autoequivalences}\label{subsec:curveautos}
Consider the subgroup $\frA_n<\sym_n$ of even permutations and the associated  
double cover: 
\[
\varpi:[E^n/\frA_n]\to[E^n/\sym_n].
\]
Then, by \cite[Lemma 5.2]{krug2014nakajima}, we know that $\varpi_*=\Ind_{\frA_n}^{\sym_n}$ is a spherical functor with cotwist $\tau^*[-1]$ and twist $\MM_{\alt}[1]$, where $\tau$ is the automorphism which interchanges the two sheets (and is represented by any transposition of $\sym_n$) of the cover and $\alt$ is the alternating representation of $\sym_n$. Therefore, by \cite[Corollary 2.4]{meachan2016note}, the left adjoint $\varpi^*=\Res_{\frA_n}^{\sym_n}$ is also spherical with cotwist $\MM_\alt[-1]$ and twist $\tau^*[1]$.

In this section, we use the spherical functor $\varpi^*$ and the fully faithful functors from Section \ref{section-ff} to construct interesting autoequivalences $\tilde{T}$ on the cover $\cD_{\frA_n}(E^n)$ which descend to give interesting autoequivalences $T$ on the base $\cD_{\sym_n}(E^n)$:
\[\xymatrix{
&& \cD_{\frA_n}(E^n)\ar@/^5pc/[d]^-{\varpi_*=\,\Ind^{\sym_n}_{\frA_n}}&&&*+[r]{\;\tilde{T}:=T_{\varpi^*i}}\ar@{|->}[d]\\
\cA\ar[rr]^-i \ar[urr]^-{\varpi^*i} && \cD_{\sym_n}(E^n) \ar[u]_-{\varpi^*=\,\Res^{\sym_n}_{\frA_n}}&&&*+[r]{T.}
}\]
More precisely, since the canonical bundle of $[E^n/\sym_n]$ has order two and $\varpi$ is unbranched, we can identify $[E^n/\frA_n]$ with the \emph{canonical cover} of $[E^n/\sym_n]$ and then our results below are obtained by applying \cite[Theorem 3.4 \& Remark 3.11]{krug2015enriques}; which is an extension (or a stacky analogue) of the results in \cite[Section 4]{bridgeland1998canonical}. 

\begin{cor}\label{descent}
If $E$ is an elliptic curve and $F:\cD(E\times E)\to\cD_{\sym_n}(E^n)$ and $\Sigma_n^*:\cD(E)\to\cD_{\sym_n}(E^n)$ are the fully faithful functors from Theorem \ref{g1-Sigma-ff} and \ref{g1-ff}, then the functors:
\[
\varpi^*F \qquad\trm{and}\qquad \varpi^*\Sigma_n^*,
\]
are spherical and the twists descend to give a new autoequivalences of $\cD_{\sym_n}(E^n)$. 
\end{cor}

\begin{proof}
This follows from \cite[Theorem 3.4 \& Remark 3.11]{krug2015enriques}. Since the cotwist $\MM_\alt[-1]$ of $\varpi^*$ is given by a shift of the Serre functor $S_{[E^n/\sym_n]}=\MM_\alt[n]$ of $\cD_{\sym_n}(E^n)$, we can apply Kuznetsov's trick \cite[Theorem 11]{addington2013categories} to conclude that if $i:\cA\hra\cD_{\sym_n}(E^n)$ is any fully faithful embedding then the composition $\varpi^*i$ is again spherical, and hence gives an autoequivalence $T_{\varpi^*i}\in\Aut( \cD_{\frA_n}(E^n))$. Now $\tau\varpi^*\simeq \varpi^*$ implies that this twist is $\tau$-invariant, which means that $\tau^*T_{\varpi^*i} \simeq T_{\varpi^*i}\tau^*$; c.f. \cite[Section 5.8]{krug2014nakajima}. Hence, one can use the decent criterion in \cite[Theorem 3.4 \& Remark 3.11]{krug2015enriques} to see that $T_{\varpi^* i}$ descends to an autoequivalence of $\cD_{\sym_n}(E^n)$.
\end{proof}

\subsection{A remark on higher dimensions}\label{subsec:higherdim}
This paper has demonstrated that the universal functor $F:\cD(A\times A)\to\cD_{\sym_n}(A^n)$ produces some interesting results in dimensions one and two. It is tempting to speculate that it should do the same in higher dimensions but we are not sure how this should work. 

Note that fully faithfulness of $\Sigma_n^*\triv_1^{\sym_n}:\cD(E)\to\cD_{\sym_n}(E^n)$ is equivalent to the structure sheaf $\cO_{[E^n/\sym_n]}$ being an exceptional object. Similarly, if $A$ is an Abelian surface then the fact \eqref{SigmaPfnctr} that $\Sigma_n^*\triv_1^{\sym_n}:\cD(A)\to\cD_{\sym_n}(A^n)$ is a $\bbP^{n-1}$-functor is equivalent to the structure sheaf $\cO_{[N/\sym_n]}$ being a $\bbP^{n-1}$-object, where $N:=\Sigma_n^{-1}(0)$; see \cite[Observation 1.2]{krug2016punits}. Indeed, one can use $(\wedge^*(\rho_n\otimes\bbC^2))^{\sym_n}\simeq\bbC[t]/t^n$ where $\deg(t)=2$; see \cite[Lemma B.5]{scala2009cohomology}, to prove this on the equivariant side.

Given that the algebra structure of our functor is \[(\wedge^*\rho_n)^{\sym_n}=\IC\simeq\H^*(\Gr(0,n),\bbC)\] in dimension one and \[(\wedge^*(\rho_n\otimes \IC^2))^{\sym_n}\simeq\H^*(\bbP^{n-1},\bbC)\simeq\H^*(\Gr(1,n),\bbC)\] in dimension two, we naively guessed that the algebra structure \[(\wedge^*(\rho_n\otimes \bbC^3))^{\sym_n}\qquad \trm{might coincide with}\qquad \H^*(\Gr(2,n),\bbC)\] in dimension three. However, this cannot be true because Macaulay tells us that the dimensions (1,0,3,1,6,3,10,6,15,0,0,0,0) of the invariants $(\wedge^i(\rho_5 \otimes \bbC^3))^{\sym_5}$ do not agree with the dimensions (1,1,2,2,2,1,1) of the cohomology groups $\H^i(\Gr(2,5),\bbC)$.

\bibliographystyle{alpha}
\bibliography{ref}
\end{document}